\documentclass{siamart220329}

\usepackage{array}
\usepackage{lipsum}
\usepackage{amsfonts}
\usepackage{graphicx}
\usepackage{epstopdf}
\usepackage{pifont}
\usepackage{enumitem}

\ifpdf
\DeclareGraphicsExtensions{.eps,.pdf,.png,.jpg}
\else
\DeclareGraphicsExtensions{.eps}
\fi

\newsiamremark{remark}{Remark}
\newsiamremark{hypothesis}{Hypothesis}
\crefname{hypothesis}{Hypothesis}{Hypotheses}
\newsiamthm{claim}{Claim}
\newsiamthm{assumption}{Assumption}
\newsiamthm{example}{Example}

\newcommand{\echo}[1]{{\color{blue}}}

\graphicspath{{./}{figures/}{../figures/}}

\headers{Universal Approximation \& Universal Interpolation}{Yongqiang Cai And Yifei Duan}

\title{Achieving Universal Approximation and Universal Interpolation via Nonlinearity of Control Families
\thanks{
    Submitted to the editors DATE.
    \funding{This work was partially supported by the National Natural Science Foundation of China under grants 12494543, the National Key R\&D Program of China under grants 2024YFF0505501, and the Fundamental Research Funds for the Central Universities.
    }
}
}

\author{
Yongqiang Cai\thanks{School of Mathematical Sciences, Laboratory of Mathematics and Complex Systems, MOE, Beijing Normal University, 100875 Beijing, China. (caiyq.math@bnu.edu.cn, yfduan@mail.bnu.edu.cn)}
\and
Yifei Duan\footnotemark[2]
}

\usepackage{amsopn}
\DeclareMathOperator{\diag}{diag}

\begin{document}

\maketitle

\begin{abstract}
A significant connection exists between the controllability of dynamical systems and the approximation capabilities of neural networks, where residual networks and vanilla feedforward neural networks can both be regarded as numerical discretizations of the flow maps of dynamical systems. Leveraging the expressive power of neural networks, prior works have explored various control families $\mathcal{F}$ that enable the flow maps of dynamical systems to achieve either the universal approximation property (UAP) or the universal interpolation property (UIP). For example, the control family $\mathcal{F}_\text{ass}({\mathrm{ReLU}})$, consisting of affine maps together with a specific nonlinear function ReLU, achieves UAP; while the affine-invariant nonlinear control family $\mathcal{F}_{\mathrm{aff}}(f)$ containing a nonlinear function $f$ achieves UIP. However, UAP and UIP are generally not equivalent, and thus typically need to be studied separately with different techniques.
In this paper, we investigate more general control families, including $\mathcal{F}_\text{ass}(f)$ with nonlinear functions $f$ beyond ReLU, the diagonal affine-invariant family $\mathcal{F}_{\mathrm{diag}}(f)$, and UAP for orientation-preserving diffeomorphisms under the uniform norm. We show that in certain special cases, UAP and UIP are indeed equivalent, whereas in the general case, we introduce the notion of local UIP (a substantially weaker version of UIP) and prove that the combination of UAP and local UIP implies UIP.
In particular, the control family $\mathcal{F}_\text{ass}({\mathrm{ReLU}})$ achieves the UIP.
\end{abstract}



\section{Introduction}

The rapid development of deep learning has opened up new research problems and opportunities in the field of control theory \cite{Li2018optimal,Benning2019Deep,Liu2019Deep,Bensoussan2022Machine,Agrachev2022Control}. In particular, the remarkable performance of residual-type neural networks in practice \cite{He2016Deep,He2016Identity}, together with their close connection to dynamical systems \cite{Weinan2017A, haber2017stable}, has motivated the study of how dynamical systems with control parameters can exhibit rich expressive power \cite{Li2022Deep,Tabuada2022Universal,Ruiz2022Interpolation,Kang2022Feedforward,Ruiz-Balet2023Neural,Alvarez2024Interplay}. This expressive power is not limited to the controllability of a single point or finitely many points, but also extends to the approximate controllability of all points within a given region. 
The controllability of a single point has been extensively studied in classical control theory, where typical tools rely on the Chow--Rashevskii theorem and the bracket-generating condition \cite{Montgomery2002tour}. For the controllability of finitely many points, also known as universal interpolation \cite{Cuchiero2020Deep,Cheng2023Interpolation} or ensemble controllability \cite{Agrachev2016Ensemble,Agrachev2020Control}, the problem can often be reduced to a single-point controllability problem in a higher-dimensional space, so that classical tools remain applicable to some extent. However, when one considers the controllability of all points in a region, the challenge becomes substantially greater: the set of points is uncountably infinite, and exact controllability is generally unattainable. In this case, alternative tools are required, particularly those from approximation theory, to properly characterize and study approximate controllability.

Consider the dynamical system
\begin{align}
    \label{eq:ode_f_theta}
 x'(t) = f_{\theta(t)}(x(t)), \quad x(0) = x_0,
\end{align}
where $\theta(t)$ is a piecewise constant control parameter and $f_\theta$ is drawn from a control family $\mathcal{F} \subset C(\mathbb{R}^d, \mathbb{R}^d)$. We are interested in the flow map $\phi_f^t(x_0)$, which sends the initial condition $x_0$ to the state $x(t)$ at time $t$. When the elements of $\mathcal{F}$ possess sufficient smoothness and are Lipschitz continuous, the flow map $\phi_f^t$ is an orientation-preserving diffeomorphism, i.e., $\phi_f^t \in \mathrm{Diff}_0(\mathbb{R}^d)$.
Although the set $\mathrm{Diff}_0(\mathbb{R}^d)$ of orientation-preserving diffeomorphisms is small within the function space $C(\mathbb{R}^d, \mathbb{R}^d)$, 
it is dense in $C(\Omega, \mathbb{R}^d)$ with respect to the $L^p$ norm for any compact domain $\Omega \subset \mathbb{R}^d$, 
provided that the dimension satisfies $d \geq 2$ \cite{Brenier2003Approximation}. However, this density property does not extend to the uniform (continuous) norm. For these reasons, when studying the controllability of all points in a domain $\Omega$, it is more appropriate to describe the problem as examining whether the flow maps $\phi_f^t$ can approximate orientation-preserving diffeomorphisms arbitrarily well under the uniform norm. 
This property is referred to as the universal approximation, and it has been widely used in the field of neural networks \cite{Cybenkot1989Approximation,Hornik1989Multilayer, Leshno1993Multilayer, pinkus1999approximation}.

Existing studies \cite{Cuchiero2020Deep, Li2022Deep, Tabuada2022Universal,Cheng2023Interpolation,Duan2025Minimal} have investigated a variety of control families $\mathcal{F}$ to ensure that the associated dynamical flow maps possess the \emph{universal approximation property} (UAP) or the \emph{universal interpolation property} (UIP). When the control family $\mathcal{F}$ itself is sufficiently expressive, it is not surprising that its flow maps admit UAP or UIP. However, for more general control families $\mathcal{F}$, identifying appropriate and mild sufficient conditions remains an important research question.
Motivated by the fact that residual networks can be viewed as forward Euler discretizations of dynamical systems \cite{Weinan2017A}, researchers have considered control families of the form
\begin{align}
    \label{eq:F_aff_f}
 \mathcal{F}_{\text{aff}}(f) 
 := 
 \left\{~ x \mapsto S f(Wx+b) ~|~ (S, W, b) \in \mathbb{R}^{d\times d} \times \mathbb{R}^{d\times d} \times \mathbb{R}^d ~\right\},
\end{align}
where $f$ is either a well function \cite{Li2022Deep} or satisfies a quadratic differential equation  \cite{Tabuada2022Universal}. This family is often referred to as the \emph{affine invariant family} \cite{Li2022Deep,Cheng2023Interpolation}. 
Recently, Cheng et al. \cite{Chen2018Neural} and Duan et al. \cite{Duan2025Minimal} show that if $f$ is nonlinear and Lipschitz continuous, then the flow maps generated by $\mathcal{F}_{\text{aff}}(f)$ achieve the UIP and UAP. 

The key to this result is that the linear span of the family, $\mathrm{span}(\mathcal{F}_{\text{aff}}(f))$, forms a sufficiently rich function space which is closely tied to the universal approximation property of neural networks \cite{Duan2025Minimal}.
When the control family $\mathcal{F}$ consists of smooth functions, the expressive power of flow maps can be further enhanced via Lie brackets \cite{Jurdjevic1997Geometric,Brockett1972System,Sussmann1972Controllability}. For example, Cuchiero et al. \cite{Cuchiero2020Deep} constructed a control family consisting of only five functions to achieve UIP. Their construction used polynomial functions of degree at most two. Since the computation of Lie brackets for polynomial functions is relatively straightforward, they were able to verify that the constructed family generates all polynomial functions through Lie brackets. 
However, Lie bracket computations generally impose strong smoothness conditions, which makes it difficult to extend this approach to more general control families. 

In situations where Lie brackets are not applicable, Duan et al. \cite{Duan2025Minimal} introduced a special control family $ \mathcal{F}_{\mathrm{ass}}(\mathrm{ReLU}) $ that achieves UAP:
\begin{align}
    \label{eq:F_ass_ReLU}
 \mathcal{F}_{\mathrm{ass}}(\mathrm{ReLU}) 
 := 
    \{~ x \mapsto Ax+b ~|~ A\in \mathbb{R}^{d\times d}, b \in \mathbb{R}^d ~\} \cup \{~ \pm \mathrm{ReLU} ~\},
\end{align}
where $\mathrm{ReLU}$ denotes the rectified linear unit, which is an elementwise function and maps $x=(x_1,...,x_d)$ to $\mathrm{ReLU}(x) = (\max(x_1,0),...,\max(x_d,0))$. It is worth noting that ReLU is globally Lipschitz continuous, and Duan et al. \cite{Duan2022Vanilla} show that its flow maps can be expressed in terms of leaky-ReLU functions \cite{Maas2013Rectifier}. The construction of this control family was inspired by the observation that leaky-ReLU feedforward networks can be viewed as numerical discretizations of dynamical system flows, and compared with residual networks, they more readily preserve the diffeomorphic property of the flows \cite{Duan2022Vanilla}. 
Moreover, the control families considered by Cuchiero et al. \cite{Cuchiero2020Deep} and Duan et al. \cite{Duan2025Minimal} are subsets of a finite-dimensional function space, so the problem can be reformulated as an affine control system of the form
\begin{align}
    \label{eq:affine_control_system}
 x'(t) &= \sum_{i=1}^m u_i(t) f_i(x(t)), \quad m \in \mathbb{Z}^+,
\end{align}
where the control parameter $u=(u_1,\dots,u_m)$ is a time-dependent function. This feature distinguishes their setting from that of affine invariant families.

It is worth noting that existing literature adopts different approaches to study UAP and UIP, yet their relationship has not been thoroughly discussed. Although Cheng et al.\cite{Cheng2023Interpolation} provided a counterexample showing that UAP and UIP are not equivalent, their formulation of UAP was in terms of the $L^p$ norm rather than the uniform norm. In this paper, we focus on the expressive power of dynamical system flow maps under the uniform norm. 
Specifically, we consider more general control families in Definition~\ref{def:control_family}, including $\mathcal{F}_{\mathrm{ass}}(f)$ with nonlinear mappings $f$ that are more general than ReLU in (\ref{eq:F_ass_ReLU}), as well as diagonal affine-invariant families $\mathcal{F}_{\mathrm{diag}}(f)$ which restrict the matrices in (\ref{eq:F_aff_f}) as diagonal matrices, and the nonlinearity of $f$ plays a crucial role in achieving UAP and UIP.. 

We provide some sufficient conditions on $f$ under which $\mathcal{F}_{\mathrm{ass}}(f)$, $\mathcal{F}_{\mathrm{aff}}(f)$, and $\mathcal{F}_{\mathrm{diag}}(f)$ achieve either UAP or UIP. In particular, in the elementwise setting, it suffices for $f$ to be nonlinear in order for $\mathcal{F}_{\mathrm{ass}}(f)$ to achieve the UAP. Moreover, we show that for certain special families, such as affine-invariant control families, UAP and UIP are equivalent. For the case of general control families, we introduce the concept of \emph{local UIP}, which is substantially weaker than UIP but much easier to verify. We prove that UAP, together with local UIP, implies UIP. As a consequence, we can establish that the control family $\mathcal{F}_{\text{ass}}(\mathrm{ReLU})$ in (\ref{eq:F_ass_ReLU}) achieves the UIP.

This paper is organized as follows. In Section~\ref{sec:results}, we introduce the necessary notation and present our main results, including theorems and illustrative examples. Section~\ref{sec:proof} is devoted to the proofs of these theorems.

\section{Notations and results}
\label{sec:results}

\subsection{Control family and hypothesis space}

Consider the $d$-dimensional initial value problem:
\begin{align}\label{eq:ODE_general}
 \left\{
    \begin{aligned}
    &\dot{x}(t) = v(x(t),t), \quad t\in(0,\tau), \tau>0,\\
    &x(0)=x_0 \in \mathbb{R}^d.
    \end{aligned}
 \right.
\end{align}
where $v(x,t)$ is Lipschitz continuous in $x$ and piecewise constant in $t$, guaranteeing a unique trajectory $x(t)$ for each $x_0$. 
The map from $x_0$ to $x(\tau)$ is called a flow map, and we denote it by $\phi_{v}^\tau(x_0)$. 
If $v(x,t) = {f_{i}}(x)$ when $t$ is within the interval $\big(\sum_{j=1}^{i-1} \tau_j, \sum_{j=1}^i \tau_j \big), \tau_i \ge 0, i = 1,2,...,n$, then the flow map $\phi_{v}^\tau$ can be represented as a composition of flow maps of autonomous systems $\phi_{{f_i}}^{\tau_i}$, 
\begin{align}
    \label{eq:flow_map_composition}
    \phi_{v}^\tau(x) 
 =
    \phi_{{f}_n}^{\tau_n} \circ \cdots \circ
    \phi_{{f}_2}^{\tau_2} \circ
    \phi_{{f}_1}^{\tau_1} (x), \quad \tau = \sum_{i=1}^n \tau_i.
\end{align}
The notation here is consistent with (\ref{eq:ode_f_theta}) if the field $v(x,t)$ is parameterized as $f_{\theta}(x)$ with parameters $\theta$ depending on $t$.
Later, we consider the case in which all $v(\cdot,t), t\ge 0,$ belong to a prescribed set $\mathcal{F}$, called the control family, and define the deduced hypothesis space $\mathcal{H}(\mathcal{F})$ as the set of all compositions of flow maps $\phi^\tau_f, f \in \mathcal{F}$:
\begin{align}
    \label{eq:hypothesis_space}
 \mathcal{H}(\mathcal{F})
 =
 \left\{
    \phi_{f_n}^{\tau_n} \circ \cdots \circ
    \phi_{f_2}^{\tau_2} \circ
    \phi_{f_1}^{\tau_1} \,\middle|\,
 f_i \in \mathcal{F}, \tau_i \ge 0, n \in \mathbb{N}
 \right\}.
\end{align}
The control families of interest in this paper are those given in the following definition.
\begin{definition}[Control family \cite{Li2022Deep,Cheng2023Interpolation,Duan2025Minimal}]
	\label{def:control_family}
Let $f: \mathbb{R}^d \to \mathbb{R}^d$ be a fixed nonlinear Lipschitz continuous function. We define the following control families:
\begin{enumerate}

    \item[1)] Associated affine  control family 
    $\mathcal{F}_{\operatorname{ass}}(f) = \mathcal{F}_0  \cup \{\pm f\}$, 
 where $\mathcal{F}_0$ is the family of all affine functions, 
    \emph{i.e.}
    $\mathcal{F}_0 = \left\{x \mapsto Ax+b \,\middle|\, A\in \mathbb{R}^{d\times d}, b \in \mathbb{R}^d\right\}$;

    \item[2)] Affine invariant control family 
    $\mathcal{F}_{\operatorname{aff}}(f) 
 = \{ x \mapsto S f(Wx+b) ~|~ (S, W, b) \in \mathbb{R}^{d\times d} \times \mathbb{R}^{d\times d} \times \mathbb{R}^d \}$;

    \item[3)] Diagonal affine invariant control family 
    $\mathcal{F}_{\operatorname{diag}}(f) = 
    \{x \mapsto D f(\Lambda x+b) ~|~ 
 D \in \mathbb{R}^{d\times d},\Lambda \in \mathbb{R}^{d\times d}, b\in \mathbb{R}^d, D~\text{and}~\Lambda~\text{are diagonal}~\}$.
\end{enumerate}
\end{definition}
Here, the nonlinearity of $f$ is essential, since the hypothesis space associated with an affine control family is relatively trivial. In fact, the hypothesis space $\mathcal{H}(\mathcal{F}_0)$ contains orientation-preserving affine functions \cite{Duan2025Minimal}:
\begin{align}
 \mathcal{H}(\mathcal{F}_0) = \left\{ x \mapsto Wx+b \,\middle|\, W \in \mathbb{R}^{d \times d},\; \det W > 0,\; b \in \mathbb{R}^d \right\}.
\end{align}
Next, we present some definitions of the nonlinearity we are interested in:
\begin{definition}[Nonlinearity \cite{Duan2025Minimal}]
 For a (vector-valued) continuous function $f=(f_1,...,f_{d'}) \in C(\mathbb{R}^d,\mathbb{R}^{d'})$ from $\mathbb{R}^d$ to $\mathbb{R}^{d'}$ with components $f_i \in C(\mathbb{R}^d), i= 1,...,d'$, we say that
\begin{enumerate}
    \item[1)] $f$ is nonlinear if at least one component $f_i$ is nonlinear,
    \item[2)] $f_i$ is coordinate nonlinear if for any coordinate vector $e_j \in \mathbb{R}^d, j=1,...,d$, there is a vector $b_{ij} \in \mathbb{R}^d$ such that $f_i(\cdot e_{j} + b_{ij}) \in C(\mathbb{R})$ is nonlinear,
    \item[3)] and $f$ is fully coordinate nonlinear if all components $f_i$ are coordinate nonlinear.
\end{enumerate}
\end{definition}

\subsection{Universal approximation and universal interpolation}

Our main results are for the expressive power of the hypothesis space $\mathcal{H}(\mathcal{F})$, where the notation of UAP and UIP is needed. We first give the formal definition of UAP below.
\begin{definition}[UAP \cite{Cai2023Achieve,Cai2024Vocabulary,Duan2025Minimal}]
 For any compact domain $\Omega \subset \mathbb{R}^d$, target function space $\mathcal{T}$ and hypothesis space $\mathcal{H}$, we say that
    \begin{itemize}
        \item[1)] $\mathcal{H}$ has the $C(\Omega)$-UAP for $\mathcal{T}$ if for any $g \in \mathcal{T}$ and $\varepsilon>0$, there is a function $h \in \mathcal{H}$ such that $\|g-h\|_{C(\Omega)} \le \varepsilon$, \emph{i.e.}
        \begin{align*}
            \|g(x)-h(x)\| \le \varepsilon, \quad \forall x \in \Omega.
        \end{align*}
        \item[2)] $\mathcal{H}$ has the $L^p(\Omega)$-UAP for $\mathcal{T}$ with $p \in [1,+\infty)$ if for any ${g} \in \mathcal{T}$ and $\varepsilon>0$, there is a function $h \in \mathcal{H}$ such that
        \begin{align*} \|{g}-h\|_{L^p(\Omega)} &= 
 \Big(\int_\Omega \|{g}(x)-h(x)\|^p dx\Big)^{1/p} \le \varepsilon.
        \end{align*}
    \end{itemize}
\end{definition}
In the absence of ambiguity, we can omit $\Omega$ to say the ``$C$-UAP'' and ``$L^p$-UAP''. Clearly, since our approximation is considered on compact domains, $C$-UAP implies $L^p$-UAP. In addition, our primary focus is on the target space $\mathrm{Diff}_0(\mathbb{R}^d)$, and in the context of this paper, we say that a control family $\mathcal{F}$ achieves the UAP if $\mathcal{H}(\mathcal{F})$ has the $C$-UAP for $\mathrm{Diff}_0(\mathbb{R}^d)$. For instance, the control family $\mathcal{F}_{\text{ass}}(\operatorname{ReLU})$ achieves the UAP.
An important feature of $C$-UAP is that it is inherited via function compositions. We formalize this property in the following proposition.
\begin{proposition}[\cite{Duan2022Vanilla}]
    \label{prop:composition_approximation}
 Let the map $T = F_n \circ ... \circ F_1$ be a composition of $n$ continuous functions $F_i$ defined on open domains $D_i$, and let $\mathcal{F}$ be a continuous function class that can uniformly approximate each $F_i$ on any compact domain $\mathcal{K}_i \subset D_i$. 
 Then, for any compact domain $\mathcal{K} \subset D_1$ and $\varepsilon >0$, there are $n$ functions $\tilde F_1, ..., \tilde F_n$ in $\mathcal{F}$ such that $\|T(x) - \tilde F_n \circ ... \circ \tilde F_1 (x)\| \le \varepsilon$ for all $x \in \mathcal{K}$.
\end{proposition}

Next, we define the UIP following \cite{Cuchiero2020Deep,Ruiz-Balet2023Neural,Cheng2023Interpolation} and introduce a new concept called the local UIP.

\begin{definition} [UIP and local UIP] 
We say the hypothesis space $\mathcal{H}$ has the UIP if for any positive integer $N$ and any $N$ distinct data points $\left(x_1, y_1\right), \ldots,\left(x_N, y_N\right) \in \mathbb{R}^d \times \mathbb{R}^d$ with $x_i \neq x_j, y_i \neq y_j$ for all $i \neq j$ there exists $\varphi \in \mathcal{H}$ such that
\begin{align}\label{eq:UIP}
    \varphi\left(x_i\right)=y_i, \quad i=1, \ldots, N .
\end{align}
Additionally, we say $\mathcal{H}$ has local UIP if for any positive integer $N$, there exists $N$ distinct points $x_1,\cdots,x_N \in \mathbb{R}^d$ and $\delta>0$ such that for any $N$ distinct points $y_1,\cdots,y_N \in \mathbb{R}^d$ with $\|y_i - x_i\| < \delta$, there exists $\varphi \in \mathcal{H}$ such that Eq.~\eqref{eq:UIP} holds.
\end{definition}
Clearly, local UIP is weaker than UIP, since it only requires the existence of a set of points at which local perturbations enable interpolation.

Note that we only discuss the UIP for dimensions larger than one, the crucial difference for the one-dimensional case is that the flow maps are always increasing, if any two data pairs $(x_i,y_i)$ and $(x_j,y_j)$ satisfy $x_i < x_j$ and $y_i > y_j$, then the UIP of the flow maps cannot hold. For this reason, when we talk about the one-dimensional UIP or local UIP in this paper, we assume the data $\{(x_i,y_i)\}_{i=1}^N$ to satisfy that $x_1 < \cdots < x_N$ and $y_1 < \cdots < y_N$.

\subsection{Main results}

Now we provided some sufficient conditions on the control families such that the deduced hypothesis spaces have UAP or UIP.

As stated before, the ReLU associated affine control family $\mathcal{F}_{\text{ass}}(\text{ReLU})$ achieves the UAP \cite{Duan2025Minimal}. Our first result extends ReLU to more general nonlinear functions. 
\begin{theorem}[UAP of associated affine control families]
    \label{th:UAP_of_F_ass}
 The hypothesis space $\mathcal{H}(\mathcal{F}_{\text{ass}}(f))$ possesses $C$-UAP for $\operatorname{Diff}_0(\mathbb{R}^d)$ on any compact set $\Omega \subset \mathbb{R}^d$ provided $f\in C(\mathbb{R}^d,\mathbb{R}^d)$ satisfies one of the following conditions: 
    \begin{enumerate}
        \item[1)] $f$ is coordinate-separable, globally Lipschitz continuous and nonlinear;
        \item[2)] $f$ is coordinate-separable, locally Lipschitz continuous, nonlinear, and in the case $d = 1$, $f$ is not a quadratic function of the form $ax^2 + bx + c$;
        \item[3)] 
        $f$ is continuously differentiable and $L^1$ integrable, 
        $\|x\|^d \|\nabla f(x)\|$ is bounded, and $\int_{\mathbb{R}^d}f(x)dx \neq 0$.
    \end{enumerate}
\end{theorem}
Here, $f$ is called coordinate-separable if $f(x) = (f_1(x_1), \dots, f_d(x_d))$ for all $x=(x_1,\cdots,x_d)\in \mathbb{R}^d$ where each $f_i$ is scalar function of one variable; And the special case where all $f_i$ are identical is referred to as an elementwise function.

The first condition in Theorem~\ref{th:UAP_of_F_ass} requires $f$ to be globally Lipschitz continuous, which ensures the global existence and uniqueness of the flow maps $\phi_f^\tau$ for all $\tau \in \mathbb{R}$. The assumption that $f$ is coordinate-separable is motivated by activation functions in neural networks, where the non-polynomial activation is applied in an elementwise manner \cite{Cybenkot1989Approximation,Hornik1989Multilayer, Leshno1993Multilayer, pinkus1999approximation}. In this setting, a remarkably simple result holds: for $\mathcal{F}_{\mathrm{ass}}(f)$, the necessary and sufficient condition to achieve UAP is that $f$ be nonlinear. The necessity follows from the fact that when $f$ is affine, the family $\mathcal{F}_{\mathrm{ass}}(f)$ reduces to the control family $\mathcal{F}_0$ mentioned earlier, whose flow maps are themselves affine functions and therefore cannot achieve UAP.

The second condition in Theorem~\ref{th:UAP_of_F_ass} relaxes globally Lipschitz continuity to local Lipschitz continuity. In this case, the flow map $\phi_f^\tau$ may exhibit finite-time blow-up, so it is defined only on a bounded set, and the time horizon $\tau$ cannot be arbitrarily large. Nevertheless, since our notion of UAP is considered only on compact sets, this restriction on the flow maps does not pose any essential difficulty.
Thus, it may appear that the nonlinearity of $f$ alone is sufficient. However, there is a special case: when the dimension equals one and $f$ is a quadratic function. In this situation, the flow maps are fractional linear transformations, also called Möbius transformations \cite[Chap. 2, Sec. 1.4]{Ahlfors1979Complex}, and their compositions with affine functions remain Möbius transformations, which prevents the realization of UAP. Excluding this exceptional case, any other nonlinear $f$ indeed enables $\mathcal{F}_{\mathrm{ass}}(f)$ to achieve UAP.

The third condition in Theorem~\ref{th:UAP_of_F_ass} drops the requirement that $f$ be coordinate-separable, which significantly increases the analytical difficulty. In this case, even if $f$ is globally Lipschitz continuous and nonlinear, this alone is not sufficient to guarantee UAP. For instance, consider the following two-dimensional example $f$: which merely permutes the two components of the ReLU function. This simple modification already causes $\mathcal{F}_{\mathrm{ass}}(f)$ to lose the ability to achieve UAP.
\begin{example}
    \label{example:permute_ReLU}
 Let function $f$ be the map $(x_1,x_2) \mapsto (\mathrm{ReLU}(x_2),\mathrm{ReLU}(x_1))$, then the associated affine control family $\mathcal{F}_{\mathrm{ass}}(f)$ cannot achieve the UAP.
\end{example}
The key observation is that in this example, the divergence of $f(x)$ vanishes everywhere, which implies that the Jacobian determinant of the flow map $\phi_f^\tau$ is constant and independent of $x$. Consequently, the Jacobian determinants of all functions in the hypothesis space $\mathcal{H}(\mathcal{F}_{\mathrm{ass}}(f))$ are also constant. This prevents the approximation of orientation-preserving diffeomorphisms whose Jacobian determinants are variable.
The third condition requires that $f$ is $L^1$-integrable and that its gradient decays at infinity. These are indeed rather strong assumptions. We believe that weaker conditions can be formulated, and we leave this question for future research. A typical example satisfying the current assumptions is given below.
\begin{example}
    \label{example:Gaussian}
 If each component of $f$ is Gaussian, then the associated affine control family $\mathcal{F}_{\mathrm{ass}}(f)$ achieves the UAP.
\end{example}

Our second result concerns the expressivity power of diagonal affine invariant control families.
\begin{theorem}[UIP of diagonal affine invariant control families]
    \label{th:UIP_of_F_diag}
 Assume $f:\mathbb{R}^d \to \mathbb{R}^d$ is globally Lipschitz continuous, and if
    \begin{enumerate}
    \item[1)] $f$ is nonlinear, then $\mathcal{F}_{\operatorname{aff}}(f)$ achieves UAP and UIP.
    \item[2)] $f$ is fully coordinate nonlinear, then $\mathcal{F}_{\operatorname{diag}}(f)$ achieves UAP and UIP.
    \end{enumerate}
\end{theorem}
For comparison, the first part of Theorem~\ref{th:UIP_of_F_diag} presents the UAP and UIP results for affine invariant control families, which summarize existing works. Specifically, for a nonlinear globally Lipschitz continuous function $f$, Duan et al. \cite{Duan2025Minimal} proved that $\mathcal{F}_{\mathrm{aff}}(f)$ achieves the UAP, while Cheng et al. \cite{Cheng2023Interpolation} established that $\mathcal{F}_{\mathrm{aff}}(f)$ achieves the UIP. By combining these results, one readily observes that under the assumption that $f$ is globally Lipschitz continuous, the UAP and the UIP for $\mathcal{F}_{\mathrm{aff}}(f)$ are in fact equivalent, since in both cases the necessary and sufficient condition is the nonlinearity of $f$.

In the case of diagonal affine invariance, Duan et al. \cite{Duan2025Minimal} proved that when $f$ is fully coordinate nonlinear, $\mathcal{F}_{\mathrm{diag}}(f)$ achieves the UAP. As for whether $\mathcal{F}_{\mathrm{diag}}(f)$ can achieve the UIP, Cheng et al. \cite{Cheng2023Interpolation} provided a sufficient condition from the perspective of the Fourier transform of $f$: each component $f_j$ must admit a spectral point $\xi_j=(\xi_{j1},...,\xi_{jd})$ that does not lie on any coordinate hyperplane, \emph{i.e.} the product $\xi_{j1} \cdots \xi_{jd} \ne 0$. This condition is both difficult to verify and rather strong. In fact, we can show that any $f$ satisfying this condition must necessarily be fully coordinate nonlinear; see Property~\ref{prop:fourier_fully_coordinate_nonlinear} for details. Moreover, the following two-dimensional example demonstrates that there exist functions that are fully coordinate nonlinear but do not satisfy the condition given by Cheng et al. Therefore, fully coordinate nonlinear functions contain a broader class of functions.
\begin{example}
    \label{example:coordinate_nonlinear}
 Let $f=(f_1,f_1) \in C(\mathbb{R}^2,\mathbb{R}^2)$ where $f_1(x_1, x_2) = \sin x_1 + \sin x_2$ is a scalar function whose Fourier transformation $\hat{f}_1$ is supported only on four points: 
 \begin{align}
	 \operatorname{supp}(\hat{f}_1) = \{(1,0), (-1,0), (0,1), (0,-1)\} \subset \mathbb{R}^2.
 \end{align}
 Then $f$ is fully coordinate nonlinear and the diagonal affine invariant control family $\mathcal{F}_{\operatorname{diag}}(f)$ achieves the UAP and UIP.
\end{example}
To prove that such $\mathcal{F}_{\mathrm{diag}}(f)$ achieves UIP, we no longer study UAP and UIP separately using different methods; instead, we attempt to establish a connection between them. This idea is formalized in our final main result, Theorem~\ref{th:localUIP_to_UIP}, which states that once UAP holds, it suffices to verify the more easily checkable local UIP to guarantee UIP.

\begin{theorem}[$C$-UAP + local UIP $\Rightarrow$ UIP]
    \label{th:localUIP_to_UIP}
 For a symmetric control family $\mathcal{F}$, if  $\mathcal{H}(\mathcal{F})$ has $C$-UAP for $\mathrm{Diff}_0(\mathbb{R}^d)$ and has local UIP,
 then $\mathcal{F}$ achieves UIP.
\end{theorem}
Here, ``symmetric'' means that if $f \in \mathcal{F}$, then $-f \in \mathcal{F}$ as well. This ensures that the hypothesis space $\mathcal{H}(\mathcal{F})$ is closed under function inversion. Exploiting this observation, we can reduce a flow map that interpolates between given data points to a composition of two flow maps, each achieving local interpolation, 
where the $C$-UAP guarantees that points can be mapped to the local neighborhoods of the target points.

Theorem~\ref{th:localUIP_to_UIP} establishes, to some extent, a connection between UAP and UIP, which benefits from our notion of UAP considered in the uniform norm. If one were to use the $L^p$ norm instead, such a connection generally does not hold, as counterexamples can be constructed to show that $L^p$-UAP and UIP are not equivalent \cite{Cheng2023Interpolation}. 
Moreover, if we only assume $L^p$-UAP and local UIP, it is unclear whether UIP follows, since $L^p$-UAP does not guarantee that each point can be mapped into the local neighborhood of its corresponding target point.
Theorem~\ref{th:localUIP_to_UIP} is the key to the proof of Theorem~\ref{th:UIP_of_F_diag} since the local UIP of $\mathcal{H}(\mathcal{F}_{\mathrm{diag}}(f))$ is easier to verify. 
Moreover, we can leverage local UIP to establish the UIP for the associated affine control family $\mathcal{F}_{\mathrm{ass}}(\mathrm{ReLU})$ as follows.
\begin{corollary}
    \label{th:UIP_ReLU}
 The associated affine control family $\mathcal{F}_{\mathrm{ass}}(\mathrm{ReLU})$ achieves UIP.
\end{corollary}

\subsection{Discussion on the difference between $\mathcal{F}_{\mathrm{ass}}(f)$ and $\mathcal{F}_{\mathrm{aff}}(f)$}

The control families considered in this paper are all related to affine functions together with a nonlinear function $f$.  
In $\mathcal{F}_{\mathrm{ass}}(f)$, the function $f$ is directly included in the control family, and its interaction with affine functions is established through the composition of flow maps. For instance, one can verify that the following function $\psi$ belongs to $\mathcal{H}(\mathcal{F}_{\mathrm{ass}}(f))$: 
\begin{align}
    \label{eq:psi_flowmap}
    \psi(x) = S \phi^\tau_f( W x + b), \quad \{S,W\} \subset \mathbb{R}^{d \times d}, b \in \mathbb{R}^d, \tau \in \mathbb{R},
\end{align}
where the matrices $S$ and $W$ have positive determinants. In contrast, in $\mathcal{F}_{\mathrm{aff}}(f)$, the function $f$ is first composed with affine functions before being added into the control family:
\begin{align}
    \label{eq:g_SfWxb}
 g(\cdot) = S f( W \cdot + b) \in \mathcal{F}_{\mathrm{aff}}(f), \quad \{S,W\} \subset \mathbb{R}^{d \times d}, b \in \mathbb{R}^d,
\end{align}
where the matrices $S$ and $W$ are not required to have positive determinants; they can be arbitrary square matrices.
In what follows, we devote some attention to discussing the distinctions and connections between these two formulations.

First, let us recall the motivation behind introducing these two types of control families.  The works of Li et al. \cite{Li2022Deep} and Tabuada et al. \cite{Tabuada2022Universal} both consider affine-invariant control families,  where vector fields of the form (\ref{eq:g_SfWxb}) are studied because their flow maps can be viewed as a continuous analogue of residual networks (ResNet).  Thus, $\mathcal{F}_{\mathrm{aff}}(f)$ may be regarded as a \emph{ResNet-type} family.  
In contrast, the flow maps of the associated affine control family are given by compositions of functions of the form (\ref{eq:psi_flowmap}),  which correspond structurally to a feedforward neural network (FNN), with the flow map $\phi_f^\tau$ playing the role of the activation function \cite{Duan2022Vanilla,Duan2025Minimal}. Therefore, $\mathcal{F}_{\mathrm{ass}}(f)$ can be interpreted as an \emph{FNN-type} family.  
In summary, both control families have direct connections to neural networks, with the main difference lying in the position where the nonlinearity $f$ is introduced.

Next, we turn to the difference in expressive power between the two families.  
A comparison of Theorems~\ref{th:UAP_of_F_ass} and Theorem~\ref{th:UIP_of_F_diag} suggests that for a fixed nonlinear $f$, $\mathcal{H}(\mathcal{F}_{\mathrm{aff}}(f))$ is in fact more expressive than $\mathcal{H}(\mathcal{F}_{\mathrm{ass}}(f))$,  
since $\mathcal{F}_{\mathrm{aff}}(f)$ is essentially larger than $\mathcal{F}_{\mathrm{ass}}(f)$. Indeed, $f \in \mathcal{F}_{\mathrm{aff}}(f)$, and any affine transformation $Ax+b$ can be approximated by functions of the form  
$
S f(Wx+b'),
$
where the approximation is possible because $f$ admits a local affine approximation in regions where it is differentiable.  
By choosing $S$, $W$, and $b'$ appropriately, this local affine approximation can be magnified to cover arbitrarily large regions,  thus ensuring that affine functions are well-approximated within $\mathcal{F}_{\mathrm{aff}}(f)$.
To observe this difference more concretely, let us revisit the nonlinear function $f$ in Example~\ref{example:permute_ReLU}. Recall that the associated affine family $\mathcal{F}_{\mathrm{ass}}(f)$ fails to achieve the UAP.  In contrast, Theorem~\ref{th:UIP_of_F_diag} guarantees that the affine invariant family $\mathcal{F}_{\mathrm{aff}}(f)$ does achieve both the UAP and the UIP.  Moreover, the following example shows that even a reduced version of $\mathcal{F}_{\mathrm{aff}}(f)$ already suffices to obtain the UAP and UIP. 
\begin{example}
    \label{example:simple_F_aff}
 Consider the function $f$ in Example~\ref{example:permute_ReLU}, we have that the control family $\mathcal{F}=\{\pm f(Ax+b) ~|~ A\in\mathbb{R}^{2\times 2}, b\in \mathbb{R}^2\}$ achieves the UAP and UIP.
\end{example}

Finally, we point out that $\mathcal{F}_{\mathrm{ass}}(f)$ can be contained in a $m$-dimensional function space $\mathcal{F}$ where $m=d^2+d+1$. In fact, the set of basis functions of $\mathcal{F}$ can be chosen as following set $\mathcal{B}$: 
\begin{align}
 \mathcal{B} = \{ x \mapsto E_{ij}x ~|~ i,j=1,...,d \} ~\cup~
    \{ x \mapsto e_i ~|~ i=1,...,d \} ~\cup~ \{f\},
\end{align}
where $E_{ij}$ denotes the square matrix with a single nonzero entry $1$ at the $(i,j)$-th position, and $e_i$ denotes the $i$-th standard basis vector. As a consequence, analyzing the expressive power of $\mathcal{H}(\mathcal{F}_{\mathrm{ass}}(f))$ can be reformulated as studying the controllability of the affine control system~(\ref{eq:affine_control_system}) where the $m$ functions are precisely chosen to be those contained in the set $\mathcal{B}$. 
In contrast, $\mathcal{F}_{\mathrm{aff}}(f)$ generally cannot be represented as a subset of a finite-dimensional function space;  
indeed, $\operatorname{span}(\mathcal{F}_{\mathrm{aff}}(f))$ is infinite-dimensional.  Consequently, the expressive power of $\mathcal{H}(\mathcal{F}_{\mathrm{aff}}(f))$ cannot be reformulated as the controllability problem of the affine control system~(\ref{eq:affine_control_system}).

\section{Proofs of the Theorems}
\label{sec:proof}

In this section, we provide detailed proofs of the main theorems stated in Section~\ref{sec:results}. Before presenting the proofs of the theorems individually, we first outline the overall strategy and key ideas that support our approach.

\subsection{Proof strategy and overview}

Since the proof of UIP will be reduced to establishing local UIP together with UAP, and local UIP is comparatively easier to verify, we will mainly focus here on the techniques used in proving UAP.

As mentioned earlier, when the control family $\mathcal{F}$ already possesses sufficiently rich expressive power, for instance $\mathcal{F}$ is dense in $C(\mathcal{K},\mathbb{R}^d)$ for every compact domain $\mathcal{K}\subset \mathbb{R}^d$, the hypothesis space $\mathcal{H}(\mathcal{F})$ can approximate any orientation-preserving diffeomorphism arbitrarily well. Thus, to prove that a general control family $\mathcal{F}$ achieves the UAP, a natural idea is to choose an essentially larger reference family $\mathcal{F}^*$ such that $\mathcal{H}(\mathcal{F}^*)$ has the UAP for $\mathrm{Diff}_0(\mathbb{R}^d)$, while every element of $\mathcal{H}({\mathcal{F}^*})$ can in turn be approximated arbitrarily well by elements of $\mathcal{H}(\mathcal{F})$.
Our proof strategy essentially follows this idea: we augment the control family $\mathcal{F}$ step by step in accordance with this principle.

Since the hypothesis spaces $\mathcal{H}(\mathcal{F})$ and $\mathcal{H}({\mathcal{F}^*})$ are both constructed through function compositions, and the approximation property of compositions is guaranteed by Proposition~\ref{prop:composition_approximation}, it suffices to verify that for any $f^* \in \mathcal{F}^*$ and $\tau>0$, the flow map $\phi_{f^*}^\tau$ belongs to the closure of $\mathcal{H}(\mathcal{F})$ under the topology of $C(\Omega)$ for any compact set $\Omega\subset \mathbb{R}^d$.
For the reference family $\mathcal{F}^*$, the result of Duan et al. \cite{Duan2025Minimal} provides a convenient and easily verifiable choice, namely  
\begin{align}
 \mathcal{F}^* = \mathcal{H}(\mathcal{F}_{\mathrm{ass}}(\mathrm{ReLU})).
\end{align}
The flow maps associated with this control family are simple, and we will use them as the foundation for the subsequent proofs in this paper.
Specifically, when proving that $\mathcal{F}_{\mathrm{ass}}(f)$ achieves the UAP, it suffices to show that the flow mapping $\phi^{\tau}_{\mathrm{ReLU}}$ can be arbitrarily well approximated by elements of $\mathcal{H}(\mathcal{F}_{\mathrm{ass}}(f))$.  
For $\mathcal{F}_{\mathrm{aff}}(f)$, we additionally need to approximate $\phi^{\tau}_{Ax+b}$, but this can be achieved simply by locally zooming in the function $f$. Therefore, the essential difficulty still lies in the approximation of $\phi^{\tau}_{\mathrm{ReLU}}$.

The key step of our proof is to construct elements in $\mathcal{H}(\mathcal{F})$ that approximate $\phi^{\tau}_{\mathrm{ReLU}}$.  We summarize below a collection of useful techniques that will be employed in this construction. We refer to these techniques collectively as the \emph{augmentation trick}, since they are all designed to augment the expressive power of compositions of the form $\phi_f^t$, with $f \in \mathcal{F}$.

\textbf{(1) Augmentation via basic properties of flow map.}

First, we can directly use the properties of flow maps to augment the control family. Let $s,a \in \mathbb{R}$, $A \in \mathbb{R}^{d \times d}$, and $b \in \mathbb{R}^d$, then we have the following relationship between flow maps: 
\begin{align}
    \phi_g^t &= 
 (a^{-1} \cdot -a^{-2}b)\circ \phi_g^{s t} \circ (a \cdot+a^{-1}b),
    & g(x)& = s f(a x + b), \quad a \neq 0,
    \\
    \phi^t_h &= 
 (A^{-1} \cdot -A^{-2}b)\circ \phi_h^{t} \circ (A \cdot+A^{-1}b),
    &h(x) &= A^{-1} f(Ax + b), \quad \det A \neq 0.
\end{align}
These relations suggest that by applying suitable translations and scalings to $f$, their flow maps can be represented as compositions of several such mappings. 
It should be noted that in the second relation, $A$ and $A^{-1}$ appear as a pair;  
we cannot replace $A^{-1}$ in $h$ with an arbitrary square matrix.  
This restriction is precisely why $\mathcal{F}_{\mathrm{ass}}(f)$ differs from $\mathcal{F}_{\mathrm{aff}}(f)$.

As an application of the first relation, we can directly show that when $f$ is the elementwise softplus function,  
$
f(x) = \ln(1 + e^x),
$  
the associated affine family $\mathcal{F}_{\mathrm{ass}}(f)$ achieves the UAP.  
This is because the softplus function can approximate ReLU arbitrarily well under suitable scaling: $a^{-1} f(ax)$ tends to $\mathrm{ReLU}(x)$ as $a$ tends to infinity. 
For the second relation, we give a two-dimensional example to show the usage. Consider $f$ as the function $(x_1,x_2) \mapsto (\mathrm{ReLU}(x_1),0)$ which modify the second element of ReLU to zero, \emph{i.e.} $f = \mathrm{ReLU}\circ E_{11} = E_{11} \circ \mathrm{ReLU}$. We can choose a special matrix $A$ such that $A^{-1} f(Ax) = E_{22} \circ \mathrm{ReLU}(x)$,
\begin{align}
    \label{eq:example_A_2d}
 A 
    & = \left( \begin{matrix}
        0 & 1 \\
 -1 & 0
    \end{matrix}\right),
 A^{-1} = -A,
 A^{-1} f(Ax) = 
 \left( \begin{matrix}
        0  \\
 \mathrm{ReLU}(x_2)
    \end{matrix}\right).
\end{align}
Combined with the linear span techniques introduced later,  
and noting that $E_{11} \circ \mathrm{ReLU} + E_{22} \circ \mathrm{ReLU} = \mathrm{ReLU}$ and $A,A^{-1} \in \mathcal{H}(\mathcal{F}_0)$,  
we can conclude that in this case $\mathcal{F}_{\mathrm{ass}}(f)$ also attains the UAP.

\textbf{(2) Augmentation via linear span.}

The second technique concerns the nested composition of two flow maps.  According to the general theory of operator splitting \cite{McLachlan2002Splitting}, the flow map of $f_1 + f_2$, $\phi_{f_1 + f_2}^\tau$, can be approximated by nested compositions of the flow maps of $f_1$ and $f_2$, namely $\phi_{f_1}^t$ and $\phi_{f_2}^t$ with some small $t$. This approximation can be extended to any finite sum of functions.  Therefore, to approximate $\phi_{\mathrm{ReLU}}^\tau$, it suffices to choose a finite set of functions $f_i$ whose sum can adequately approximate ReLU.

For example, using the well-known Fourier series theory \cite{Stein2011Fourier}, ReLU can be approximated on any large region by a finite sum of trigonometric functions. Combined with the scaling and translation relations from the first augmentation technique, it is straightforward to show that both $\mathcal{F}_{\mathrm{ass}}(\sin)$ and $\mathcal{F}_{\mathrm{ass}}(\cos)$ achieve the UAP,  where $\sin$ and $\cos$ act elementwise on the argument.

In addition, combining the first augmentation trick, we can find the hypothesis space $\mathcal{H}(\hat{\mathcal{F}}_{\mathrm{ass}}(f))$ of the following control family $\hat{\mathcal{F}}_{\mathrm{ass}}(f)$ has the same expressivity with $\mathcal{H}(\mathcal{F}_{\mathrm{ass}}(f))$: 
\begin{align}
    \label{eq:hat_F_ass_f}
\hat{\mathcal{F}}_{\mathrm{ass}}(f)
:=
\left\{
 x \to
\sum_{i=1}^{n} s_i\,g(a_i x + b_i)
\;\middle|\;
s_i\in \mathbb{R}, a_i\ge 0, b_i \in \mathbb{R}^d, n \in \mathbb{Z}^+, g \in \mathcal{F}_{\mathrm{ass}}(f)
\right\}.
\end{align}
Here the function $s_i g(a_i x + b_i)$ is a special case of $s A^{-1} f(Ax + b)$ with $A = a_i I$ and $s = s_i a_i$.

\textbf{(3) Augmentation by Lie brackets.}

When $\mathcal{F}$ contains smooth functions, we can also utilize Lie brackets to enhance the expressiveness of the hypothesis space. Define $\text{Lie}(\mathcal{F})$ as the span of all vector fields of $\mathcal{F}$ and their iterated Lie brackets (of any order):
\begin{align*}
 \text{Lie}(\mathcal{F})
    &=
 \text{span}\{
 f_1, [f_1,f_2],[f_1,[f_2,f_3]],... |
 f_1, f_2, f_3,... \in \mathcal{F}
        \} \notag
\end{align*}
where the operator $[f,g](x) = \nabla g(x) f(x) - \nabla f(x) g(x)$ is the Lie bracket between two smooth vector fields $f$ and $g$. One basic property is that the flow map $\phi^\tau_{[f,g]}$ of $[f,g]$ can be approximated by
$
\phi^{\sqrt{\tau}}_{-g} \circ
\phi^{\sqrt{\tau}}_{-f} \circ
\phi^{\sqrt{\tau}}_{g} \circ
\phi^{\sqrt{\tau}}_{f}
$
for $\tau$ small enough.
Therefore, when $\mathrm{Lie}(\mathcal{F})$ is easy to compute, to verify whether $\mathcal{F}$ achieves the UAP, it suffices to check whether $\mathrm{Lie}(\mathcal{F})$ can approximate ReLU and affine functions arbitrarily well.  

For example, in the case where $f$ consists of polynomial functions, one can calculate the following Lie brackets:
\begin{align}
    \label{eq:Lie_polynomial_mn}
 [x^n, x^m]=(m-n)x^{n+m-1}, n,m \in \mathbb{Z}^+,
\end{align}
where $x^n$ denotes the elementwise $n$-th power function applied to $x$.
It is easy to see that when $n \ge 3$, 
$\mathrm{Lie}(\{1, x^n\}) = \mathrm{span}\{1,x,x^2,\dots\}$ 
contains all elementwise polynomials. Therefore, by the Stone-Weierstrass theorem, the ReLU can be arbitrarily approximated, and we conclude that $\mathcal{F}_{\mathrm{ass}}(x^n), n \ge 3,$ achieves the UAP.

\subsection{Useful lemmas}

To make the arguments in the outlined strategy more rigorous, we now formally present the aforementioned augmentation techniques in the form of lemmas.

The first lemma shows that the flow maps of two dynamical systems are arbitrarily close if the corresponding vector fields are sufficiently close. 

\begin{lemma}[Lemma 4 of \cite{Duan2022Vanilla}]
    \label{lemma:ODE_error_estimation}
 Consider two ODE systems
    \begin{align}\label{eq:ODE_system}
 \dot{x}(t) = f_i(x(t)), \quad t\in(0,\tau), \quad i=1,2,
    \end{align}
 where the $f_i(x)$ are continuous in $x\in\mathbb{R}^d$. In addition, we assume that $f_1(x)$ is globally $L$-Lipschitz continuous, \emph{i.e.,} $\|f_1(x)-f_1(x')\| \le L \|x-x'\|$ for any $x,x' \in \mathbb{R}^d$. Then, for any compact domain $\Omega$ and $\varepsilon>0$, there exists $\delta \in (0,1]$ such that the flow maps $\phi_{f_1}^\tau(x) ,\phi_{f_2}^\tau(x)$ satisfy
    \begin{align}
        \|\phi_{f_1}^\tau(x)- \phi_{f_2}^\tau(x)\| \le \varepsilon,
 \quad \forall x\in\Omega,
    \end{align}
 provided that $\|f_1(x)-f_2(x)\|< \delta$ for all $x \in \Omega_\tau$, where $\Omega_\tau$ is a compact domain defined as $\Omega_\tau = \{ x + (V+1) \tau e^{L\tau} x' ~|~ x \in \Omega, \|x'\| \le 1, V = \max_{x\in \Omega}\{\|f_1(x)\|\} \}$.
\end{lemma}

The second lemma ensures that the augmentation technique via the linear span of Lipschitz continuous control families. 

\begin{lemma}
    \label{lemma:flow_sum_f}
    \label{th:flow_sum_f}
 Let $\mathcal{F}=\{f_1,\cdots f_m\}$ be a finite control family, where $f_i:\mathbb{R}^d\to \mathbb{R}^d$ denotes locally Lipschitz functions; then, for any compact domain $\Omega \subset \mathbb{R}^d$, any function
    $f = \sum_{i=1}^m a_i f_i$ with $a_i \in \mathbb{R}^+$, and $\tau \ge 0$, the flow map $\phi_f^\tau$ (if well defined on $\Omega$) is in the closure of $\mathcal{H}(\mathcal{F})$; 
    \emph{i.e.}, $\phi_f^\tau \in {\operatorname{cl} (\mathcal{H}(\mathcal{F}))}$ under the topology of $C(\Omega)$.
\end{lemma}
\begin{proof}

The proof of the globally Lipschitz case can be found in \cite{Duan2025Minimal} by the splitting method.  
Let  
$$
T_{n,j} = \big(\phi_{f_m}^{a_m\tau/n} \circ \cdots \circ \phi_{f_1}^{a_1\tau/n}\big)^{\circ j},
$$
where the superscript $\circ j$ denotes the $j$-fold composition. 
Then the flow map of a linear combination of functions in $\mathcal{F}$ can be approximated by compositions of the individual flow maps: for any $\varepsilon > 0$, there exists $n > 0$ such that
\begin{align}\label{eq:flow_sum_f_approximation}
 \left\| \phi_f^\tau(x) - T_{n,n}(x) \right\|_\infty < \varepsilon, \quad \forall\, x \in \Omega.
\end{align}

For the locally Lipschitz case, assumming $\phi_f^\tau$ is well-defined on $\Omega$, we can choose an integer $N$ large enough or $\tau/N$ small enough, and a larger open ball $\mathcal{O} \subset \mathbb{R}^d$ containing all vectors
$\phi_f^{{j\tau}/{n}}(\Omega)$ and $T_{n,j}(\Omega)$ for all relevant $j=1,...,n$ with $n <N$.
Let $L_{\mathcal{O}}$ be a shared Lipschitz constant for $f$ and all $f_i$ on $\mathcal{O}$, then modify the definition of $f_i$ outside $\mathcal{O}$ such that the modified functions is globally Lipschitz with constant $L_{\mathcal{O}}$. Then we can treat the approximation as the globally Lipschitz continuous case and finish the proof.
\end{proof}

The third lemma considers the approximation for flow maps of Lie brackets, which follows the Baker--Campbell--Hausdorff expansion (see \cite{Muger2019Notes,Yoshida1990Construction} for example).
\begin{proposition}
    \label{prop:one-order-Lie}
    \label{th:flow_approximation_Lie_f1_f2}
Let $f_1$ and $f_2$ be smooth and hence locally Lipschitz continuous, and assume the flow $\phi_{[f_1,f_2]}^t$ is well-defined on an open set $D\subset\mathbb R^d$ for all $t\in[0,\tau], \tau >0$.
Then for any compact domain $K\subset D$, there exist a positive integer $N$ and two positive constants $C_K,L_K>0$ (depending on $K,f_1,f_2,\tau$) such that the following approximation estimation holds for all $n > N$,
\begin{align}
 \bigl\|\phi_{[f_1,f_2]}^{\tau} - (\Psi_{\Delta t})^{\circ n}\bigr\|_{C(K)}
\le 
C_K \tau e^{L_K \tau} \sqrt{\Delta t}
= O(\sqrt{\Delta t}),
\end{align}
where $\Delta t=\tau/n$ and $\Psi_{\Delta t}$ is the following function composed by flow maps of $f_1,f_2$,
\begin{align}
    \Psi_{\Delta t} = \phi_{f_2}^{\sqrt \Delta t}\circ\phi_{f_1}^{\sqrt \Delta t}\circ\phi_{-f_2}^{\sqrt \Delta t}\circ\phi_{-f_1}^{\sqrt \Delta t}.
\end{align}

\end{proposition}

\begin{proof}

Denote $g = [f_1,f_2]$, then the set $\left\{ \phi_g^{\,t}(x) \;\middle|\; x \in K,\; t \in [0,\tau] \right\}$ is compact and covered by a ball $B_r$ with radius $r>0$.
For any $x \in B_{r+2}$ and $\Delta t$ small enough, the Taylor expansion for $\phi_g^{\Delta t}(x)$ and $\Phi_{\Delta t}(x)$ implies the difference 
$
R(x;\Delta t):=
\phi_{g}^{\Delta t}(x) - \Psi_{\Delta t}(x) = O(\Delta t^{3/2}).
$
Therefore there exists a $\tau_0 >0$ such that $\|R(x;\Delta t)\| < 1$ for any $x \in B_{r+1}$ and $\Delta t < \tau_0$, which implies that $\Psi_{\Delta t}(x) \in B_{r+2}$.
Similarly, there exists a $\tau_1 >0$ such that all intermediate trajectory involved in the evaluation of $\Psi_{\Delta t}(x),x \in B_{r+1},$ is covered by $B_{r+2}$ provided $\Delta t < \tau_1$.
Finally, we can take $\Delta t < \min(\tau_0,\tau_1,\tau_2)$ such that all the involved trajectories started from $x \in K$ are covered by $B_{r+2}$, where $\tau_2>0$ is small enough and its existence will be clear later.

Since $f_1$, $f_2$, and $g$ are smooth, all their derivatives are locally Lipschitz continuous. Denote $L_K >0$ as the Lipschitz constant shared by $f_1$, $f_2$ and $g$ on $B_{r+2}$, and $C_K>0$ as a constant depending on the derivatives of $f_1$ and $f_2$ on $B_{r+2}$ such that
\begin{align}
    \label{eq:error_BCH}
    \|R(x;\Delta t)\|
 =
    \|\phi_{g}^{\Delta t}(x) - \Psi_{\Delta t}(x)\| < C_K \Delta t^{3/2}, 
 \quad \forall x \in B_{r+1}.
\end{align}
Define the approximation error on $K$ after $j$ steps as
\begin{align}
 E_j := \sup_{x\in K}\bigl\|\phi_{g}^{j \Delta t}(x) - \Psi_{\Delta t}^{\circ j}(x)\bigr\|,
 \quad
 j = 0,1,...,n,
\end{align}
then we have $E_0=0$ and the following estimation, according to (\ref{eq:error_BCH}) and the local Lipschitz continuity of $\phi_{g}^{\Delta t}$ 
\begin{align}
 E_{j} 
    &\le 
 \sup_{x\in K}
 \left(
 \bigl\|\phi_{g}^{j\Delta t}(x) - \phi_{g}^{\Delta t} \circ \Psi_{\Delta t}^{\circ {(j-1)}}(x)\bigr\| 
 + \bigl\|\phi_{g}^{\Delta t} \circ \Psi_{\Delta t}^{\circ {(j-1)}}(x) - \Psi_{\Delta t}^{\circ {j}}(x) \bigr\| 
 \right)\\
    & \le 
 e^{L_K \Delta t}\,E_{j-1} + C_K\,\Delta t^{3/2}.
\end{align}
Unrolling the recursion or employing the discrete Grönwall inequality  gives
\begin{align}
 E_j 
    &\le 
 (e^{L_Kh})^j E_0 + \sum\limits_{i=0}^{j-1} (e^{L_K \Delta t})^{j-1-i} C_K \Delta t^{3/2}
 = C_K \Delta t^{3/2}\frac{(e^{L_K \Delta t})^j - 1}{e^{L_K \Delta t}-1}\\
    &\le 
 \frac{C_K}{L_K} (e^{L_Kj \Delta t}-1) \sqrt{\Delta t}
 \le 
 \frac{C_K}{L_K} (e^{L_K\tau}-1) \sqrt{\Delta t} 
 \le  
 C_K \tau e^{L_K \tau} \sqrt{\Delta t}. 
    \label{eq:main_estimation_inequality}
\end{align}
Here the inequality $j \Delta t \le n \Delta t = \tau$ and $s \le e^{s} - 1 \le s e^{s}$, $s \ge 0$, are employed.

Now we can assign the prescribed $\tau_2$ as $\tau_2 = 1/(C_K \tau e^{L_K \tau})^2$ which ensures $E_j < 1$ for all $j=1,...,n$.
Taking $N = \lceil \tau / \min(\tau_0,\tau_1,\tau_2) \rceil $ and $n>N$, we have $\Delta t = \tau/n < \min(\tau_0,\tau_1,\tau_2)$.
In addition, for any $x \in K$, we have the trajectory 
$\phi_g^t(x) \in B_r$ for all $t \in [0,\tau]$, 
$\Psi_{\Delta t}^{\circ j}(x) \in B_{r+1}$ for all $j \in \{0,1,...,n\}$,
and the intermediate trajectory during $\Psi_{\Delta t}^{\circ j}(x)$ belongs to $B_{r+2}$. This ensures the whole estimation procedure and finishes the proof.
\end{proof}

Combining Lemma~\ref{lemma:flow_sum_f} and Lemma~\ref{th:flow_approximation_Lie_f1_f2}, we can derive the following lemma, which ensures the augmentation technique via Lie brackets.

\begin{lemma}
    \label{lemma:lie_bracket_error}
    \label{th:flow_Lie}
 Let $h:\mathbb{R}^d\to \mathbb{R}^d$ be globally Lipschitz continuous and $\mathcal{F}$ be a symmetric smooth control family.
 If for any compact domain $\mathcal{K}\subset \mathbb{R}^d$ and $\delta>0$, there exists $g \in \operatorname{Lie}(\mathcal{F})$ such that 
    $$
    \|h - g\|_{C(\mathcal{K})} < \delta,
    $$
 then all flow maps $\phi_h^\tau, \tau \ge 0$, are in the closure of $\mathcal{H}(\mathcal{F})$ under the topology of $C(\Omega)$ for any compact domain $\Omega\subset \mathbb{R}^d$; In other words, for any compact domain $\Omega$, $\tau\ge 0$ and $\varepsilon>0$, there is a function $\varphi \in \mathcal{H}(\mathcal{F})$ such that 
    $\|\phi_h^\tau - \varphi\|_{C(\Omega)} < \varepsilon$.
\end{lemma}

\begin{proof}
 Denote by $\mathcal{L}$ the set of all iterated Lie brackets generated by $\mathcal{F}$:
    \begin{align}
 \mathcal{L} := \left\{ [f_{1}, [f_{2}, \cdots [f_{i-1}, f_{i}]\cdots]] 
 ~|~  
 f_{j} \in \mathcal{F}, j=1,...,i, i \in \mathbb{Z}^+ \right\}.
    \end{align}
 Then $\mathcal{L}$ is also a smooth family and we have $ \operatorname{Lie}(\mathcal{F}) = \mathrm{span}(\mathcal{L})$. 
 The Lemma \ref{lemma:ODE_error_estimation} indicates that we can choose a function $g \in \operatorname{Lie}(\mathcal{F})$ 
 approximates $h$ such that the flow maps satisfy
    \begin{equation}
        \label{eq:approximation_01}
        \|\phi_h^\tau - \phi_{g}^\tau\|_{C(\Omega)} < \frac{\varepsilon}{3},
 \quad 
 g = \sum_{i=1}^n a_i P_i, ~
 P_i \in \mathcal{L}, ~ a_i\in\mathbb{R}, ~n \in \mathbb{Z}^+.
    \end{equation}
 The Lemma \ref{lemma:flow_sum_f} implies that $\phi_{g}^\tau$ can be approximated by $\mathcal{H}(\mathcal{L})$. 
 In other words, there is a large integer $m$, and $m$ positive numbers $\tau_1, ...,\tau_m >0$ and $m$ functions $p_1,...,p_m \in \mathcal{L}$ such that the following approximation holds,
     \begin{equation}
        \label{eq:approximation_02}
 \Bigl\| \phi_{g}^\tau - \phi_{p_1}^{\tau_1} \circ \cdots \circ \phi_{p_m}^{\tau_m}  \Bigr\|_{C(\Omega)} 
 < \frac{\varepsilon}{3}.
    \end{equation}

 By recursively applying Proposition~\ref{th:flow_approximation_Lie_f1_f2} and making use of Proposition~\ref{prop:composition_approximation}, we conclude that each well-defined flow map $\phi_p^t, p \in \mathcal{L}, t>0,$ can be approximated by elements of $\mathcal{H}(\mathcal{F})$. As a consequence, we can approximate $\phi_{p_1}^{\tau_1} \circ \cdots \circ \phi_{p_m}^{\tau_m}$ by an element $\varphi$ in $\mathcal{H}(\mathcal{F})$ such that $\|\phi_{p_1}^{\tau_1} \circ \cdots \circ \phi_{p_m}^{\tau_m} - \varphi \|_{C(\Omega)} < \frac{\varepsilon}{3}$. Combining the estimation in (\ref{eq:approximation_01}) and (\ref{eq:approximation_02}), we obtain the desired approximation.
\end{proof}

\subsection{Proof of Theorem \ref{th:UAP_of_F_ass}}

\echo{
\begin{theorem}[UAP of associated affine control families, Theorem \ref{th:UAP_of_F_ass}]
 The hypothesis space $\mathcal{H}(\mathcal{F}_{\text{ass}}(f))$ possesses $C$-UAP for $\operatorname{Diff}_0(\mathbb{R}^d)$ on any compact set $\Omega \subset \mathbb{R}^d$ provided $f\in C(\mathbb{R}^d,\mathbb{R}^d)$ satisfies one of the following conditions: 
    \begin{enumerate}
        \item[1)] $f$ is coordinate-separable, globally Lipschitz continuous and nonlinear;
        \item[2)] $f$ is coordinate-separable, locally Lipschitz continuous, nonlinear, and in the case $d = 1$, $f$ is not a quadratic function of the form $ax^2 + bx + c$;
        \item[3)] 
        $f$ is continuously differentiable and $L^1$ integrable, 
        $\|x\|^d \|\nabla f(x)\|$ is bounded, and $\int_{\mathbb{R}^d}f(x)dx \neq 0$.
    \end{enumerate}
\end{theorem}
}

\subsubsection{Coordinate-separable cases}

Although the coordinate-separable setting is more general than the elementwise one, we can reduce it to the elementwise case for our analysis. To justify this reduction, we first state the following proposition, which shows a basic property of nonlinear functions.

\begin{proposition}
\label{th:property_of_nonlinearity}
Let $f:\mathbb{R}\to\mathbb{R}$ be continuous and non-polynomial. 
Then there exists a constant $c \in \mathbb{R}$ such that 
$
g_c(x):=f(x+c)-f(x)
$
is non-polynomial as well.
\end{proposition}
\begin{proof}
We complete the proof by contradiction. Specifically, we aim to show that if $f$ is continuous and for every $c \in \mathbb{R}$ the function $f(x+c)-f(x)$ is a polynomial, then $f(x)$ itself must also be a polynomial. Indeed, let $f(x+c)-f(x)=P(x)$ be a polynomial. By a basic property of polynomials, there exists a polynomial $Q(x)$ such that $Q(x+c)-Q(x)=P(x)$. Define $g = f-Q$. Then $g$ is a bounded periodic function with period $c$. Hence $f = g+Q$ is the sum of a polynomial and a $c$-periodic function. Since $c$ is arbitrary, $g$ must necessarily be a constant function, and therefore $f$ is a polynomial.
\end{proof}

Next, we provide the proof for the coordinate-separable cases.  

\begin{proof}[Proof of Theorem \ref{th:UAP_of_F_ass} for condition 1) and 2)]

Since condition~1) is a special case of condition~2),  it suffices to provide a proof for condition~2) only.

We will present the proof by considering three cases. The first case is when at least one component of $f=(f_1,\dots,f_d)$ is nonlinear and non-polynomial.  The second case is when all components $f_i$ of $f$ are polynomials and at least one component has a degree larger than 2. The third case is when all components $f_i$ of $f$ are quadratic polynomials and the dimension $d \neq 1$.

(1) For the first case, without loss of generality, we assume that the first component $f_1$ is nonlinear and is non-polynomial. 
In this case, there must exist a constant $c \in \mathbb{R}$ such that  
\(
\sigma(t) := f_1(t+c) - f_1(t)
\)
is non-polynomial and locally Lipschitz continuous as a function of $t \in \mathbb{R}$ according to Proposition~\ref{th:property_of_nonlinearity}. Let $b = c e_1 \in \mathbb{R}^d$, we have the following function $\tilde f(x)$ belongs to $\hat{\mathcal{F}}_{\mathrm{ass}}(f)$ defined in (\ref{eq:hat_F_ass_f}),
\begin{align}
 \tilde f(x) = f(x+b) - f(x) = (\sigma(x_1),0,...,0).
\end{align}
Then, using the augmentation trick similar to the example in (\ref{eq:example_A_2d}), we can broadcast the function $\sigma$ to each coordinate. Therefore, the control family needs to be examined, which becomes $\mathcal{F}_{\mathrm{ass}}(\sigma)$,  where $\sigma$ acts elementwise on each coordinate, and we only need to prove that $\mathcal{F}_{\mathrm{ass}}(\sigma)$ achieves the UAP.

Since $\sigma$ is a continuous non-polynomial function, by the universal approximation theorem for single-hidden-layer neural networks \cite{Leshno1993Multilayer}, we have for any compact interval $I$, the neural network of the form $h$,
\begin{align}
 h(t) = \sum_{i=1}^{N} s_i \sigma(a_i t +b_i), ~ a_i,b_i,s_i,t \in \mathbb{R}
\end{align}
can approximate the scalar ReLU arbitrarily well. 
The parameters $a_i$ can be restricted to be nonnegative.
Now, consider applying $h$ elementwise to $x \in \mathbb{R}^d$. It follows that ReLU can be approximated arbitrarily well by $\hat{\mathcal{F}}_{\mathrm{ass}}(\sigma)$. Since all the functions constructed here are locally Lipschitz continuous, Lemma~\ref{th:flow_sum_f} can be applied to show the flow map $\phi_{\mathrm{ReLU}}^\tau, \tau \in \mathbb{R},$ can be approximated by the hypothesis space $\mathcal{H}(\hat{\mathcal{F}}_{\mathrm{ass}}(\sigma))$. 
Consequently, since $\mathcal{F}_{\mathrm{ass}}(\mathrm{ReLU})$ achieves the UAP, both $\hat{\mathcal{F}}_{\mathrm{ass}}(\sigma)$ and $\mathcal{F}_{\mathrm{ass}}(\sigma)$ also achieve the UAP.

(2) For the second case, without loss of generality, we assume that $f_1$ is the one with the largest degree among the components of $f$. In this case, we can assume $f_1$ as the form $f_1(t) = t^n + ... + c_1 t + c_0, c_i \in \mathbb{R},$ where $n$ is the degree greater or equals to 2. Here we can further assume $c_0=c_1=0$ by adding a proper affine function to $f$.

When $n\ge 3$, we consider the Lie algebra generated by $\{f, e_1\}$, where $e_1$ is the first standard basis vector in $\mathbb{R}^d$. Direct calculation according to (\ref{eq:Lie_polynomial_mn}) gives that
\begin{align}
 \text{span}\{e_1, x_1 e_1, x_1^2 e_1,x_1^3 e_1,\cdots\} \subset \mathrm{Lie}(\{f,e_1\}) \subset \mathrm{Lie}(\mathcal{F}_{\mathrm{ass}}(f)).
\end{align}
Then employing the Stone-Weierstrass theorem \cite{stone1948generalized}, we know that the function $x \mapsto (\mathrm{ReLU}(x_1),0,...,0)$ can be approximated by $\mathrm{Lie}(\mathcal{F}_{\mathrm{ass}}(f))$. 
Using the augmentation trick similar to the example in (\ref{eq:example_A_2d}) again, 
and together with Lemma~\ref{th:flow_Lie}, we can conclude that the flow map $\phi_{\mathrm{ReLU}}^\tau, \tau \in \mathbb{R},$ can be approximated by the hypothesis space $\mathcal{H}({\mathcal{F}}_{\mathrm{ass}}(f))$. Hence $\mathcal{F}_{\mathrm{ass}}(f)$ achieves the UAP.

(3) For the third case, when $n=2$ and the dimension $d\ge 2$, we need a more detailed computation of $\mathrm{Lie}(\mathcal{F}_{\mathrm{ass}}(f))$. According to the results of Cuchiero et al. \cite{Cuchiero2020Deep}, 
the Lie algebra generated by the control family $\mathcal{F} = \mathcal{F}_0 \cup \{V_4, V_5\}$ is rich enough to generate all polynomial vector fields on $\mathbb{R}^d$,
where $V_4(x)=(x_d^2,0,\cdots,0), V_5(x)=(x_1x_d,x_2x_d,\cdots,x_d^2)$. Based on this, we can compute that  
\begin{align}
 \mathrm{Lie}(\mathcal{F}_{\mathrm{ass}}(g)) = \mathrm{Lie}(\mathcal{F}),
\end{align} 
where $g(x) = x^2 = (x_1^2,\cdots,x_d^2)$ acts elementwise. In fact, let $A = (a_{ij})_{d \times d}$, we have
\begin{align*}
 [A,g](x) &= (\nabla x^2) Ax - Ax^2 
 = \left( \begin{matrix}
        2x_1(a_{11}x_1 + ... + a_{1d}x_d) - (a_{11}x_1^2 + ... + a_{1d}x_d^2) \\
        \vdots \\
        2x_d(a_{d1}x_1 + ... + a_{dd}x_d) - (a_{d1}x_1^2 + ... + a_{dd}x_d^2)
    \end{matrix}\right).
\end{align*}
Set $A=E_{id}$ and denote
$u_i(x):=[E_{id},g](x)=(2x_ix_d-x_d^2) e_i$, 
one can verify that 
\begin{align*}
    &[\tfrac{1}{2}E_{1d},u_1](x) = (x_d^2,0,\cdots,0)^T, \quad [E_{dd},u_d](x) = (0,\cdots,0,x_d^2)^T, \\
    & [\tfrac{1}{2}E_{dd}+E_{ii}, u_i](x) = x_ix_d e_i, \quad \text{for } i=2,\cdots,d-1. 
\end{align*}
Then we can conclude that $V_4,V_5\in \operatorname{Lie}(\mathcal{F}_{\mathrm{ass}}(g))$ and $\mathrm{Lie}(\mathcal{F}_{\mathrm{ass}}(g)) = \mathrm{Lie}(\mathcal{F})$.
To handle $f$ instead of $g$, we only need to observe that $[E_{11}, f](x) = x_1^2 e_1$ and once again make use of the augmentation trick similar to (\ref{eq:example_A_2d}). Consequently, $\mathcal{F}$, $\mathcal{F}_{\mathrm{ass}}(g)$ and $\mathcal{F}_{\mathrm{ass}}(f)$ achieve the UAP.

The proof is finished now.
\end{proof}

Note that the above proof does not work in the case $n=2$ and $d=1$, i.e., when 
$
f(x)=x^2, x \in \mathbb{R}.
$
In this situation, the flow map $\phi_f^\tau(x) = \frac{x}{1-\tau x}$ is well defined when $t$ is small, and any function $\phi$ in the hypothesis space $\mathcal{H}(\mathcal{F}_{\mathrm{ass}}(f))$ has the following form
\begin{align}
        \phi(x) = \frac{a x+b}{cx+d}, \quad a,b,c,d \in \mathbb{R},
\end{align}
which is known as the M\"obius transformations. The expressivity of such a hypothesis space is limited and has no UAP for general monotonic functions.

\subsubsection{Non-coordinate-separable cases}

Here, we employ some techniques to reduce the case where $f$ is not coordinate-separable to the coordinate-separable case. This reduction allows us to leverage the results already established for coordinate-separable functions.

\echo{
$f$ is continuously differentiable and $L^1$ integrable, 
        $\|x\|^d \|\nabla f(x)\|$ is bounded, and $\int_{\mathbb{R}^d}f(x)dx \neq 0$.
}

\begin{proof}[Proof of Theorem \ref{th:UAP_of_F_ass} for condition 3)]

Without loss of generality, we suppose the first component $f_1$ of $f=(f_1,...,f_d)$ has nonzero integral  $\int_{\mathbb{R}^d} f_1(x) dx \neq 0$. Then for any $x=(x_1,\cdots, x_d)\in \mathbb{R}^d$, we can define a function $\bar{f}: \mathbb{R}^d \to \mathbb{R}^d$ as follows:
\begin{align}
    \label{eq:def_bar_f}
 \bar{f}(x) 
 = \int_{\mathbb{R}^{d-1}} f(x_1,y) dy
 = \int_{\mathbb{R}^{d-1}} f \left( x+\sum_{i=2}^ds_ie_i \right) ds_2 \cdots ds_d,
\end{align}
where $y = (x_2,...,x_d)$ denotes the last $d-1$ components of $x$.
The Fubini's theorem (see \cite[Chap. 6]{bruckner1997real} for example) ensures that the $\bar f$ is well defined and integrable. In addition, $\bar{f}$ only depends on $x_1$ and hence is of the form $\bar{f}(x) = (\bar{f}_1(x_1),\cdots, \bar{f}_d(x_1))$ where $\bar{f}_1(x_1)$ has nonzero integral. It is implied that $\bar{f}_1(x_1)$ is nonlinear. 

The condition $\|x\|^d \|\nabla f(x)\|$ is bounded implies that there is a constant $C$ such that $(1 + \|x\|^d) \|\nabla f(x)\| \le C$, which means the gradient $\nabla f$ decays, ensuring that $f$ and $\bar{f}$ are globally Lipschitz continuous. 
In fact, the derivative of $\bar f$ is bounded
\begin{align}
 |\bar f'(x_1)|
 \le 
    \int_{\mathbb{R}^{d-1}} \| \nabla f(x) \| d x_2 \cdots d x_d 
 \le 
    \int_{\mathbb{R}} \frac{ C \omega_{d-2} r^{d-2}}{1+r^d} dr
 < \infty,
\end{align}
where $\omega_{d-2}$ is the volume of the unit ball in $\mathbb{R}^{d-2}$.

According to the condition 1), we can see that the control family $\mathcal{F}_{\mathrm{ass}}(g)$ achieves UAP where the function $g: x \to (\bar{f}_1(x_1),0,...,0)$ is coordinate-separable, globally Lipschitz continuous and nonlinear.
Now we only need to prove that the flow maps $\phi_{g}^t, t \in \mathbb{R}$ can be approximated by $\mathcal{H}(\mathcal{F}_{\mathrm{ass}}(f))$.
This is guaranteed by two facts: 
\begin{itemize}
    \item[(1)] $g(x)$ can be approximated by the following $g_{\varepsilon}(x)$,
\begin{align*}
 g_{\varepsilon}(x)
 :=
 A_{\varepsilon}^{-1} \bar{f} (A_{\varepsilon} x)
 = \bigl(\bar{f}_1(x_1),\;\varepsilon \bar{f}_2(x_1),\;\dots,\;\varepsilon \bar{f}_d(x_1)\bigr).
\end{align*}
where $A_{\varepsilon} = \diag(1,1/\varepsilon,...,1/\varepsilon)$ and $\varepsilon>0$ is arbitrarily small. Hence $\phi_{g}^t$ can be approximated by 
$
\phi_{g_\varepsilon}^t = A_{\varepsilon}^{-1}\circ \phi^t_{\bar{f}} \circ A_{\varepsilon} 
\in 
\mathcal{H}(\mathcal{F}_{\mathrm{ass}}(\bar f))
$

    \item[(2)] $\bar f$ can be approximated by $\hat{\mathcal{F}}_{\mathrm{ass}}(f)$ defined in (\ref{eq:hat_F_ass_f}) because $f$ decays and hence the integral in (\ref{eq:def_bar_f}) can be approximated by finite summation of shifts of $f$. As a consequence, $\phi^t_{\bar{f}}$ can be approximated by  $\mathcal{H}(\mathcal{F}_{\mathrm{ass}}(f))$.
    
\end{itemize}

The proof is complete.
\end{proof}

\subsection{Proof for Example \ref{example:permute_ReLU}}

\echo{
\begin{example}[Example \ref{example:permute_ReLU}]
 Let function $f$ be the map $(x_1,x_2) \mapsto (\mathrm{ReLU}(x_2),\mathrm{ReLU}(x_1))$, then the associated affine control family $\mathcal{F}_{\mathrm{ass}}(f)$ can not achieve the UAP.
\end{example}
}

The function $f$ in Example~\ref{example:permute_ReLU} is divergence-free. Therefore, the following lemma is sufficient to support the claim.

\begin{lemma}
 Let $f:\mathbb{R}^2 \to \mathbb{R}^2$ be globally Lipschitz continuous and divergence free, then there exists $g \in \operatorname{Diff}_0(\mathbb{R}^2)$ that can not be approximated by $\mathcal{H}(\mathcal{F}_{\mathrm{ass}}(f))$.
\end{lemma}

\begin{proof}
    
 Since the divergence of any affine map is constant and $\nabla \cdot f\equiv 0$ , we have $\det (\nabla \phi^\tau_f(x)), \det (\nabla \phi^\tau_{A\cdot+b}(x) )$ are both constant with respect to $x$. In fact, denoted $J(t,x)=\nabla \phi_f^t(x)$, the chain rule of derivatives shows $\frac{d}{dt}J(t,x) = \nabla f(\phi_f^t(x)) J(t,x)$, and the Jacobian formula \cite[Part III, Sec. 8.3]{Magnus2019Matrix}, gives: 
    \begin{align}
 \frac{d}{dt}\det J(t,x) 
 &= 
 \det J(t,x) \cdot \operatorname{tr} \big(J(t,x)^{-1} \frac{d}{dt}J(t,x)\big)\\
        & = \det J(t,x) \cdot \operatorname{tr}\big(J(t,x)^{-1}\nabla f(\phi_f^t(x)) J(t,x)\big)\\
        & = \det J(t,x) \cdot \operatorname{tr}\big(\nabla f(\phi_f^t(x))\big)\\
        & = \det J(t,x) \cdot \nabla \cdot \big(f (\phi_f^t(x))\big) = 0.
    \end{align}
 Then for every function $\phi \in \mathcal{H}(\mathcal{F}_{\mathrm{ass}}(f))$, the Jacobian determinant 
    \begin{align}
 J_\phi(x) = \det (\nabla \phi(x)), \quad x \in \mathbb{R}^2,
    \end{align}
 is also a constant function.

 To finish the proof, it is enough to verify the case of 
    $g(x)=(e^{x_1},x_2)\in \operatorname{Diff}_0(\mathbb{R}^2)$.
 We prove by contradiction that $g$ cannot be approximated.
 The key point is that the Jacobian determinant of $g$ is $J_g(x) = e^{x_1}$, which is not a constant function.

 Now suppose $g$ can be approximated by $\mathcal{H}(\mathcal{F}_{\mathrm{ass}}(f))$. 
 Take two points $P_1 = (-2,0)$ and $P_2 = (2,0)$, and let $\Omega_1$ and $\Omega_2$ denote the disks of radius $R=1$ centered at $P_1$ and $P_2$, respectively. Choose $\Omega = [-4,4]^2 \supset \Omega_1 \cup \Omega_2$.
 It is implied that for any $\varepsilon > 0$, there exists 
    $\phi \in \mathcal{H}(\mathcal{F}_{\mathrm{ass}}(f))$
 such that 
    $\|g - \phi\|_{C(\Omega)} < \varepsilon$.
 Consider the volumes of $g(\Omega_i)$ and $\phi(\Omega_i)$, which can be calculated by integrating their Jacobian determinants:
    \begin{align}
 \operatorname{V}(g(\Omega_i)) = \int_{g(\Omega_i)} d x = \int_{\Omega_i} J_g(x)\, dx, \quad 
 \operatorname{V}(\phi(\Omega_i)) = \int_{\phi(\Omega_i)} d x = \int_{\Omega_i} J_\phi(x)\, dx.
    \end{align}
 Since $g$ approximates $\phi$, the volume of $g(\Omega_i)$ is close to that of $\phi(\Omega_i)$, which can be quantitatively estimated by the Steiner formula~\cite{Steiner1840parallele,Gray2004Steiner}: 
    \begin{align}
 \left| \int_{\Omega_i} (J_g(x) - J_\phi(x))\, dx \right|  
 = 
 \left| \operatorname{V}(g(\Omega_i)) - \operatorname{V}(\phi(\Omega_i)) \right| 
 \le 
 M \varepsilon + \pi \varepsilon^2,
    \end{align}
 where $M = \max(l(\partial g(\Omega_1)),l(\partial g(\Omega_2)))$, $l(\partial g(\Omega_i))$ is the length of boundary of $ g(\Omega_i)$. 
 Notice that $J_\phi(x)$ is a constant and the volume of $\Omega_1$ and $\Omega_2$ are the same, it is implied that
    \begin{align}
        \Delta : = \left| \int_{\Omega_1} J_g(x) dx - \int_{\Omega_2} J_g(x) dx \right|  
 \le 2 M \varepsilon + 2 \pi \varepsilon^2. 
    \end{align}
 The arbitrarity of $\varepsilon$ implies $\Delta=0$, which is not true for the function $g$ we chose. Therefore, $\mathcal{H}(\mathcal{F}_{\mathrm{ass}}(f))$ cannot approximate $g$.
\end{proof}

\subsection{Proof for Example \ref{example:Gaussian}}

\echo{
\begin{example}[Example \ref{example:Gaussian}]
 If each component of $f$ is Gaussian, then the associated affine control family $\mathcal{F}_{\mathrm{ass}}(f)$ achieves the UAP.
\end{example}
}

It is obvious that any Gaussian function $f$ is continuously differentiable and $L^1$ integrable, and the integral $\int_{\mathbb{R}^d}f(x)dx \neq 0$. In addition, the gradients are also exponentially decaying, which ensures that $\|x\|^d \|\nabla f(x)\|$ is bounded. 
In addition to the Gaussian function, smooth functions with compact support also satisfy these conditions.

\subsection{Proof of Theorem \ref{th:UIP_of_F_diag}}

\echo{
\begin{theorem}[UIP of diagnoal affine invariant control families,Theorem \ref{th:UIP_of_F_diag}]
 Assume $f:\mathbb{R}^d \to \mathbb{R}^d$ is globally Lipschitz continuous, and
    \begin{enumerate}
    \item[1)] $f$ is nonlinear, then $\mathcal{F}_{\operatorname{aff}}(f)$ achieves UAP and UIP.
    \item[2)] $f$ is fully coordinate nonlinear, then $\mathcal{F}_{\operatorname{diag}}(f)$ achieves UAP and UIP.
    \end{enumerate}
\end{theorem}
}

We first recall the well-known Chow-Rashevskii theorem \cite[Chap.~5]{Agrachev2004Control} in $\mathbb{R}^d$.

\begin{proposition}[Chow-Rashevskii Theorem]
 Let $\mathcal{F}$ be a symmetric family of smooth vector fields on an open connected domain $\Omega \subset \mathbb{R}^d$. If $\left. \operatorname{Lie} (\mathcal{F}) \right|_q = \mathbb{R}^d$ for all $q \in \Omega$, then the system is controllable in $\Omega$.
\end{proposition}
Here, a special case is when $\left. \mathrm{span}(\mathcal{F})\right|_q = \mathbb{R}^2$, in which controllability still holds. This version is more suitable for control families that are not smooth.

Then the problem of simultaneously controlling $N$ distinct points, 
$x_1,\dots,x_N,$ in $\mathbb R^d$ is equivalent to the classical controllability of the lifted $Nd$‑dimensional system
\begin{align}\label{eq:Nd_system}
 \dot X_N(t)
 =\bigl(g(x_1(t)),\dots,\,g(x_N(t))\bigr) 
    \in \mathcal F_N
 ,
\end{align}
where 
$X_N = (x_1,\dots,x_N) \in \mathbb{R}^{dN}$
and 
$\mathcal F_N = \bigl\{\,(g(x_1),\dots,g(x_N))\mid g\in\mathcal{F} \bigr\}= \mathcal{F}^{\otimes N}$
is the lifted control family.
Therefore, to establish the UIP, 
it suffices to strengthen the condition 
$\operatorname{Lie}(\mathcal{F})|_{q} = \mathbb{R}^d$ 
by requiring that $\operatorname{Lie}(\mathcal{F})$ possesses the interpolation property \cite{Cuchiero2020Deep}.
That is, for any positive integer $N$, any $N$ distinct points 
$x_1, \ldots, x_N \in \mathbb{R}^d$, and any prescribed target points 
$v_1, \ldots, v_N \in \mathbb{R}^d$ (which may coincide), there exists 
$g \in \operatorname{Lie}(\mathcal{F})$ such that  
\begin{align}
 g(x_j) = v_j, \quad j = 1, \ldots, N.
\end{align}
A special case arises when $\operatorname{Lie}(\mathcal{F})$ or 
$\operatorname{span}(\mathcal{F})$ contains all polynomial vector fields.  
In this situation, the control family $\mathcal{F}$ is capable of achieving the UIP.

Now we present the proof of Theorem \ref{th:UIP_of_F_diag}.

\begin{proof}[Proof of Theorem \ref{th:UIP_of_F_diag}]

When $f=(f_1,...,f_d)$ satisfies the respective conditions, Duan et al. \cite{Duan2025Minimal} proved that both $\mathcal{F}_{\mathrm{aff}}(f)$ and $\mathcal{F}_{\mathrm{diag}}(f)$ achieve the UAP, while Cheng et al. \cite{Cheng2023Interpolation} established that $\mathcal{F}_{\mathrm{aff}}(f)$ achieves the UIP. Therefore, we only need to provide the case of $\mathcal{F}_{\mathrm{diag}}(f)$ that achieves the UIP. However, to present the proof with a unified perspective, we also give the proof idea for all cases.

For the UAP, to prove that $\mathcal{F}_{\operatorname{aff}}(f)$ or $\mathcal{F}_{\operatorname{diag}}(f)$ achieves the property, it suffices to show that $\operatorname{span}(\mathcal{F}_{\operatorname{aff}}(f))$ or $\operatorname{span}(\mathcal{F}_{\operatorname{diag}}(f))$ has sufficiently rich approximation capability. Note that the functions in $\operatorname{span}(\mathcal{F}_{\operatorname{aff}}(f))$ take the same form as single-hidden-layer neural networks, so the proof follows directly from the universal approximation theorem for neural networks. For $\mathcal{F}_{\operatorname{diag}}(f)$, although the matrices $D$ and $\Lambda$ in functions of the form $Df(\Lambda x+b)$ are restricted to be diagonal, the stronger nonlinearity assumption on $f$ is sufficient to ensure that $\operatorname{span}(\mathcal{F}_{\operatorname{diag}}(f))$ attains the desired approximation power.

For the UIP, the situation differs between $\mathcal{F}_{\operatorname{aff}}(f)$ and $\mathcal{F}_{\operatorname{diag}}(f)$. For $\mathcal{F}_{\operatorname{aff}}(f)$, since the matrices $S$ and $W$ in functions of the form $S f(Wx+b)$ are unrestricted, it is straightforward to show that $\operatorname{span}(\mathcal{F}_{\operatorname{aff}}(f))$ possesses a global interpolation property, thereby directly yielding the UIP. In contrast, for $\mathcal{F}_{\operatorname{diag}}(f)$, the matrices involved are diagonal, which limits us to establishing only a local interpolation property for $\operatorname{span}(\mathcal{F}_{\operatorname{diag}}(f))$. To obtain the UIP in this case, we exploit the fact that $\mathcal{H}(\mathcal{F}_{\operatorname{diag}}(f))$ satisfies the UAP, and then invoke Theorem~\ref{th:localUIP_to_UIP} to complete the proof.

Thus, it remains to show that $\mathcal{F}_{\operatorname{diag}}(f)$ achieves the local UIP.  
In fact, we will prove the interpolation property of $\operatorname{span}(\mathcal{F}_{\operatorname{diag}}(f))$ on the following connected open subset $\Omega_N$ of $\mathbb R^{dN}$ for any $N \in \mathbb{Z}^+$, 
\begin{align}
    \label{eq:Omega_N}
    \Omega_N
 =\bigl\{(x_1,\dots,x_N)\in\mathbb R^{dN}\,\bigm|\,
 x_j^{(k)}<x_{j'}^{(k)}\text{ for all }1\le j<j' \le N,\;1\le k\le d\bigr\}.
\end{align}
Here, each coordinate of  $x_j = \bigl(x^{(1)}_j, \ldots, x^{(d)}_j\bigr)$ increases with respect to $j$.
The connectness and openness of $\Omega_N$ are easy to verify, and the interpolation property means that for any $X_N\in\Omega_N$, we have
\begin{align}
    \label{eq:span_condition_dN}
 \operatorname{span} \mathcal{F}_N(X_N)=\mathbb{R}^{dN},
 \quad
 \mathcal{F}_N = \mathcal{F}_{\mathrm{diag}}(f) ^{\otimes N}.
\end{align}
We will use a proof by contradiction to establish equation (\ref{eq:span_condition_dN}), following the main ideas of Cheng et al \cite{Cheng2023Interpolation}. Once this equality is established, we can invoke the Chow–Rashevskii theorem to obtain the controllability of the control family $\mathcal{F}_N$ on $\Omega_N$, and the openness of $\Omega_N$ allows us to select $N$ distinct points that satisfy the requirements to achieve the local UIP.

Now we prove the equation (\ref{eq:span_condition_dN}). 
Suppose there exists $X_N=(x_1,\cdots,x_N)\in \Omega_N$ such that the linear subspace
$\operatorname{span}(\mathcal{F}_N(X_N))$ is not the whole space $\mathbb{R}^{dN}$, then there exists a nonzero normal vector $c \in \mathbb{R}^{dN}$ which is orthogonal to $\operatorname{span}(\mathcal{F}_N(X_N))$.
In other words, there exist $dN$ real numbers labeled by $c^{(k)}_j, 1\le j\le N, 1\le k\le d,$ with at least one $c^{(k)}_j$ nonzero, such that 
\begin{align}
    \label{eq:span_condition_zero}
    \sum_{k=1}^d\sum_{j=1}^N c^{(k)}_j g_k(x_j) = 0, \quad \forall g=(g_1,\cdots,g_d)\in \mathcal{F}_{\mathrm{diag}}(f).
\end{align}
Assume one nonzero componnet of $c$ is $c^{(K)}_{J}\neq 0$, and choose the function $g \in \mathcal{F}_{\operatorname{diag}}(f)$ as $g(x) = E_{KK} f(\Lambda x+b) = e_K f_K(\Lambda x+b)$, then we have
\begin{align}
    \sum_{j=1}^N c^{(K)}_j f_K(\Lambda x_j+b) = 0
\end{align}
holds for all diagonal matrix $\Lambda\in \mathbb{R}^{d\times d}$ and $b\in \mathbb{R}^d$.
Note that $f$ is globally Lipschitz and can be regarded as a tempered distribution. Taking the Fourier transform with respect to $b$, we have
\begin{align}
    \sum_{j=1}^N c^{(K)}_j e^{\mathrm{i} \xi^T \Lambda x_j} \widehat f_K(\xi) = 0, \quad \forall \xi\in \mathbb{R}^d
\end{align}
holds for all diagonal matrix $\Lambda\in \mathbb{R}^{d\times d}$, where $\widehat f_K$ is the Fourier transform of $f_K$.

Since $f_K$ is globally Lipschitz, there is a nonzero spectral vector $\xi=(\xi_1,\cdots,\xi_d)\in \operatorname{supp} \widehat{f}$, 
otherwise $\widehat{f}$ would be supported only at the origin, implying that $f$ is linear or polynomial, contradicting our assumption. 
Without loss of generality, we may assume $\xi_1\neq 0$ and let $\Lambda = s E_{11}$,  $s \in \mathbb{R}$, then 
\begin{align}\label{eq:span_condition_zero_exp}
    \sum_{j=1}^N c^{(K)}_j e^{\mathrm{i} s \xi_1 x_j^{(1)}} = 0
\end{align}
holds for any $s \in \mathbb{R}$. Since each $\{\xi_1 x_j^{(1)}\}$ is different from each other, the functions $e^{\mathrm{i} s \xi_1 x_j^{(1)}}, j=1,\cdots,N,$ are linearly independent. It follows that $c^{(K)}_j = 0$ for all $j = 1, \dots, N$, which is a contradiction for the assumption that $c^{(K)}_{J}\neq 0$. 
Thus, the equation (\ref{eq:span_condition_dN}) is implied, and the proof is complete.
\end{proof}

\begin{remark}
In the above proof, we considered the set $\Omega_N$ in equation~(\ref{eq:Omega_N}) under rather strong constraints; this choice is crucial for the subsequent steps.  
We note that using the function from Example~\ref{example:coordinate_nonlinear}
\begin{align}
 f\colon (x^{(1)},x^{(2)}) \mapsto \bigl(\sin x^{(1)} + \sin x^{(2)},\;\sin x^{(1)} + \sin x^{(2)}\bigr),
\end{align}
in the above argument prevents the subsequent steps from being carried out.
For instance, when $N=4$, take the four points
\begin{align}
 x_1=(1,1),\quad x_2=(-1,1),\quad x_3=(-1,-1),\quad x_4=(1,-1).
\end{align}
Then there exists a nonzero vector $c$ with 
$c^{(1)}=(c^{(1)}_1,c^{(1)}_2,c^{(1)}_3,c^{(1)}_4)=(1,-1,1,-1)$ and $c^{(2)}_j=0, j=1,...,4$, 
such that (\ref{eq:span_condition_zero}) holds,
which invalidates the required nondegeneracy in the proof and thus prevents the argument from proceeding.
\end{remark}

\begin{remark}
 In Theorem \ref{th:UIP_of_F_diag}, we assumed $f$ is fully coordinate nonlinear to ensure $\mathcal{F}_{\operatorname{diag}}(f)$ achieves UAP and UIP. In fact, if each component of $f$ is nonlinear, it is enough for local UIP.  
\end{remark}

\subsection{Proof for Example \ref{example:coordinate_nonlinear}}

\echo{
\begin{example}[Example \ref{example:coordinate_nonlinear}]
 Let $f=(f_1,f_1) \in C(\mathbb{R}^2,\mathbb{R}^2)$ where $f_1(x_1, x_2) = \sin x_1 + \sin x_2$ is a scalar function whose Fourier transformation $\hat{f}_1$ is supported only on four points: 
    $$
 \operatorname{supp}(\hat{f}_1) = \{(1,0), (-1,0), (0,1), (0,-1)\} \subset \mathbb{R}^2.
    $$
 Then $f$ is fully coordinate nonlinear and the diagonal affine invariant control family $\mathcal{F}_{\operatorname{diag}}(f)$ achieves the UAP and UIP.
\end{example}
}

The function in Example \ref{example:coordinate_nonlinear} is fully coordinate nonlinear, yet it does not satisfy the conditions proposed by Cheng et al. \cite{Cheng2023Interpolation}. Generally, we present a proposition showing that our assumption on $f$ (being fully coordinate nonlinear) includes a strictly larger class of functions.
\begin{proposition}
    \label{prop:fourier_fully_coordinate_nonlinear}
 Suppose $f=(f_1,\cdots,f_d)\in C(\mathbb{R}^d,\mathbb{R}^d)$ is globally Lipschitz and satisfies that for each $j=1,\cdots,d$, there exists some $(\xi_{j1},\cdots,\xi_{jd})\in \operatorname{supp} \widehat f_j$ such that $\xi_{j1}\cdots \xi_{jd} \neq 0$, then $f$ is fully coordinate nonlinear.
\end{proposition}

\begin{proof}
We prove the contradictory proposition. Suppose $f$ is not fully coordinate nonlinear. Then there exist an index $j\in\{1,\dots,d\}$ and a direction $e_k$ such that for every shift $b\in\mathbb R^d$ the slice
$\phi_b(t)=f_j(te_k+b)$
is an affine function with the form $a\,t+c$ where $a,c\in \mathbb{R}$ depends on $b$.
For notational simplicity in the subsequent derivations, we may assume that $e_k$ is the last standard basis vector $e_d$. The general case follows by permuting the coordinates.

The Fourier transform of the scalar function $\phi_b: t \to a t +c$ is $\widehat\phi_b$,
\begin{align}
    \label{eq:fourier_affine}
 \widehat\phi_b(\omega)=\int_{\mathbb R}(a\,t+c)\,e^{-\mathrm{i}\omega t}\,dt
 =2\pi \mathrm{i}\,a\,\delta'(\omega)+2\pi c\,\delta(\omega), \omega \in \mathbb{R}, 
\end{align}
which vanishes for all $\omega\neq0$, \emph{i.e.} $\operatorname{supp}(\widehat{\phi}_b) = \{0\}$ for all $b\in\mathbb R^d$.
By appling the relation between Fourier transform of $\phi$ and $f_j$, similar to the Fourier slice theorem \cite{Bracewell1956Strip} , we have
\begin{align}
 \hat\phi_b(\omega) 
    &= 
    \int_{\mathbb{R}} f_j(t e_d+b) e^{-\mathrm{i}\omega t} \, dt 
 	=  
    \int_{\mathbb{R}^d} f_j(x) e^{-\mathrm{i}\omega (x_d-b_d)} 
    \prod\limits_{l =1}^{d-1} \delta(x_l-b_l) \, dx \\ 
    &= 
	\tfrac{1}{(2\pi)^{d}} 
    \int_{\mathbb{R}^d}\int_{\mathbb{R}^d} 
 \widehat{f}_j(\xi) e^{\mathrm{i}\xi\cdot x} 
 e^{-\mathrm{i}\omega (x_d-b_d)} 
    \prod\limits_{l =1}^{d-1} \delta(x_l-b_l) 
    \, d\xi \, dx \\ 
    &= 
	\tfrac{1}{(2\pi)^{d}}
    \int_{\mathbb{R}^d} \widehat{f}_j(\xi) e^{\mathrm{i}\xi\cdot b} 
 \left( 
        \int_{\mathbb{R}^d} 
 e^{\mathrm{i}(\xi_d-\omega)(x_d-b_d)} 
        \prod\limits_{l =1}^{d-1} \delta(x_l-b_l) 
        \, dx 
 \right) 
 d\xi \\ 
    &= 
	\tfrac{1}{(2\pi)^{d}}
    \int_{\mathbb{R}^d} \widehat{f}_j(\xi) e^{\mathrm{i}\xi\cdot b} 
    \, 2\pi \delta(\omega-\xi_d)\, d\xi \\ 
    &= 
	\tfrac{1}{(2\pi)^{d-1}}
 \left.
 \left(
        \int_{\mathbb{R}^{d-1}} 
 \widehat{f}_j(\xi) e^{\mathrm{i}\xi \cdot b} 
        \, d\xi_1 \cdots d\xi_{d-1} 
 \right)
 \right|_{\xi_d=\omega} 
    \label{eq:Fourier_slice}.
\end{align}
The last equality implies that, when $\hat\phi_b(\omega) $ is regarded as a function of $b$, it is precisely the inverse Fourier transform of $\widehat{f}_j$ with respect to the first $d-1$ variables.

If there is a spectral point $\xi \in \operatorname{supp} (\widehat f_j)$ has the last componnet $\xi_d = \omega$, then the equality in (\ref{eq:Fourier_slice}) implies that, when viewed as a function of $b$, $\hat\phi_b(\omega) $ is not identically zero.
That is, there exists some $b \in \mathbb{R}^d$ such that $\omega \in \operatorname{supp}(\widehat{\phi}_b) = \{0\}$, which is possible only when $\omega=0$. 
Therefore, we have
\begin{align}
 \operatorname{supp}\widehat f_j \subset \{\xi\in\mathbb R^d ~|~ \xi_k=0\},
\end{align}
so in particular there is no $\xi \in \operatorname{supp}\widehat f_j$ with $\xi_1\cdots\xi_d\neq0$. This completes the proof.
\end{proof}

\subsection{Proof of Theorem \ref{th:localUIP_to_UIP} and Corollary \ref{th:UIP_ReLU}}

\echo{
\begin{theorem}[$C$-UAP + local UIP $\Rightarrow$ UIP, Theorem \ref{th:localUIP_to_UIP}]
 For a symmetric control family $\mathcal{F}$, if  $\mathcal{H}(\mathcal{F})$ has $C$-UAP for $\mathrm{Diff}_0(\mathbb{R}^d)$ and has local UIP,
 then $\mathcal{H}(\mathcal{F})$ has UIP.
\end{theorem}
}
\begin{proof}[Proof of Theorem \ref{th:localUIP_to_UIP}]
 Recall that $\mathcal{H}(\mathcal{F})$ holds $C$-UAP and local UIP, for any distinct data pairs $\{(x_1,y_1),\cdots,(x_N,y_N)\}$, where $ x_i,y_i \in \mathbb{R}^d$ satisfy $x_i\neq x_j, y_i \neq y_j$ for any $i\neq j$, our goal is to find a flow map $\Phi \in \mathcal{H}(\mathcal{F})$ such that $\Phi(x_i)=y_i$ for all $i=1,\cdots,N$.

 Since $\mathcal{H}(\mathcal{F})$ has local UIP, then there exists $z_1,\dots,z_N \in \mathbb{R}^d$ and $\delta > 0$ such that for any $\{z_i'\}_{i=1,\cdots,N}$ and $\{z_i''\}_{i=1,\cdots,N}$ satisfying $z_i', z_i'' \in B(z_i, \delta)$, there exists a function $\phi_3, \phi_4 \in \mathcal{H}(\mathcal{F})$ such that $\phi_3(z_i) = z_i'$ and $\phi_4(z_i) = z_i''$ for all $i = 1,\dots,N.$ According to the first part of Theorem \ref{th:UIP_of_F_diag} we know that there exist $\psi_1,\psi_2 \in \operatorname{Diff}_0(\mathbb{R}^d)$ such that $\psi_1(x_i)=z_i=\psi_2(y_i)$ for all $i=1,\cdots,N$. Since $\mathcal{H}(\mathcal{F})$ has $C$-UAP for $\text{Diff}_0(\mathbb{R}^d)$, we can find flow maps $\phi_1,\phi_2 \in \mathcal{H}(\mathcal{F})$ such that 
    \begin{align}
        \phi_1(x_i), \phi_2(y_i) \in B(z_i,\delta).
    \end{align}
 Asign $z_i' = \phi_1(x_i)$ and $z_i'' = \phi_2(y_i)$, we have 
    \begin{align}
        \phi_3^{-1}\circ \phi_1(x_i) = z_i = \phi_4^{-1} \circ \phi_2(y_i), \quad i=1,\cdots,N.
    \end{align}
 Let $\Phi:=\phi_2^{-1} \circ \phi_4 \circ \phi_3^{-1}\circ \phi_1$, since $\mathcal{F}$ is symmetric, we have $\phi_2^{-1},\phi_3^{-1}\in \mathcal{H}(\mathcal{F})$, and $\Phi \in \mathcal{H}(\mathcal{F})$ is what we need, the proof is complete.
\end{proof}
The proof idea is summarized in Figure \ref{fig:UIP_proof}:
\begin{figure}[htb!]
    \centering
    \includegraphics[width=0.73\textwidth]{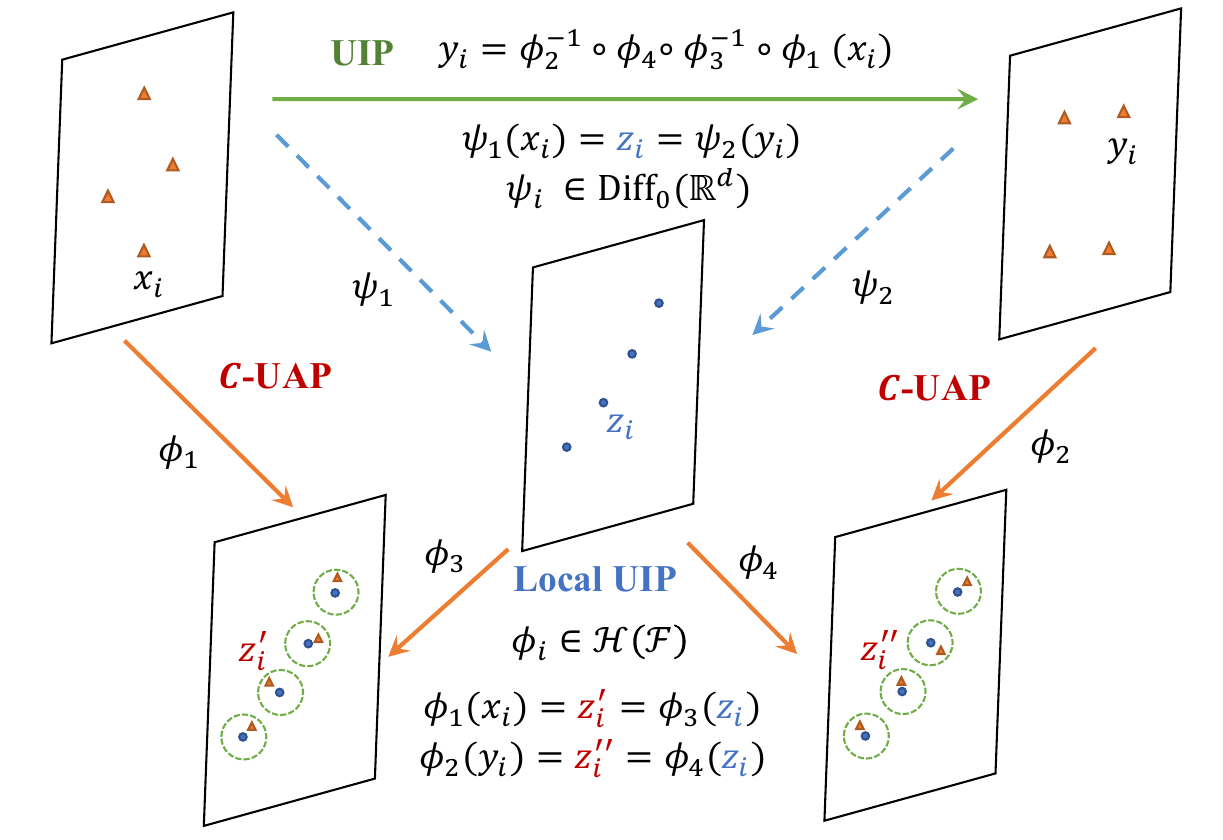}
    \caption{Proof idea for Theorem \ref{th:localUIP_to_UIP}. 
 }
    \label{fig:UIP_proof}
\end{figure}

\echo{
\begin{corollary}[Corollary \ref{th:UIP_ReLU}]
 The associated affine control family $\mathcal{F}_{\mathrm{ass}}(\mathrm{ReLU})$ achieves UIP.
\end{corollary}
}
\begin{proof}[Proof of Corollary \ref{th:UIP_ReLU}]
Since the control family $\mathcal{F}_{\mathrm{ass}}(\mathrm{ReLU})$ achieves the UAP, here we only need to check that it achieves the local UIP. For any positive integer $N$, we can choose $N$ points $x_i, i=1,...,N,$ in $\mathbb{R}^d$ as $x_i = i e_1$, which are arranged along the first axis. Then for any corresponding target points $y_i$ such that $\|y_i-x_i\| < \delta = 1/(2N)$, we can construct a flowmap $\Phi$ in $\mathcal{H}(\mathcal{F}_{\mathrm{ass}}(\mathrm{ReLU}))$ such that $\Phi(x_i)=y_i$. 
It is easy to construct such a $\Phi$ step by step because the flowmap of ReLU is the leaky-ReLU function, which keeps the points in one direction fixed while moving only the points in the other direction \cite{Duan2022Vanilla}.
In fact, since the chosen points follow a regular pattern, we can inductively construct $\Phi_i$ such that 
$y_j = \Phi_i(x_j)$ for $j = 1, \ldots, i$.  
Then $\Phi_{i+1}$ only needs to keep the first $i$ points fixed while moving the subsequent points so that 
$y_{i+1} = \Phi(x_{i+1})$.  
In this way, by finally setting 
\(
\Phi = \Phi_N \circ \cdots \circ \Phi_1,
\)
the desired mapping is obtained.

\end{proof}

\subsection{Proof for Example \ref{example:simple_F_aff}}

\echo{
\begin{example}[Example \ref{example:simple_F_aff}]
 Consider the function $f$ in Example~\ref{example:permute_ReLU}, we have the control family $\mathcal{F}=\{\pm f(Ax+b) ~|~ A\in\mathbb{R}^{2\times 2}, b\in \mathbb{R}^2\}$ achieves the UAP and UIP.
\end{example}
}

Note that the control family $\mathcal{F}=\{\pm f(Ax+b) ~|~ A\in\mathbb{R}^{2\times 2}, b\in \mathbb{R}^2\}$ here differs slightly from $\mathcal{F}_{\mathrm{aff}}(f)$.  
In particular, the functions in $\mathcal{F}_{\mathrm{aff}}(f)$ are of the form 
\[
g(x) = \pm S f(Ax+b).
\]
But in our current setting, the matrix $S$ is constrained to be the identity.  
This restriction, however, can be eliminated through an appropriate technique.  
To see this, one may directly verify that the function $f: (x_1,x_2) \mapsto (\mathrm{ReLU}(x_2),\mathrm{ReLU}(x_1))$ still ensures the validity of the following identity:
\begin{align}
S \operatorname{ReLU}(Ax+b)
=
\sum_{i=1}^2 \sum_{j=1}^2 s_{ij} f \big(
 (E_{12}+E_{21})(E_{ij}(Ax+b))
 \big),
\end{align}
where $s_{ij}$ is the $(i,j)$ component of $S$. 
This implies that 
\(
\operatorname{span}(\mathcal{F}) = \operatorname{span}\bigl(\mathcal{F}_{\mathrm{aff}}(\mathrm{ReLU})\bigr),
\)
and therefore $\mathcal{F}$ achieves both UAP and UIP, since $\mathcal{F}_{\mathrm{aff}}(\mathrm{ReLU})$ does.
This stands in sharp contrast to Example~\ref{example:permute_ReLU}, where it was shown that $\mathcal{F}_{\mathrm{ass}}(f)$ can not achieve the UAP.

\section*{Acknowledgments}
We would like to express our sincere gratitude to Zuowei Shen, Qianxiao Li, Ting Lin and Jingpu Cheng for the insightful discussions.

\bibliographystyle{siamplain}
\bibliography{refs.bib}

\begin{thebibliography}{10}

\bibitem{Agrachev2016Ensemble}
{\sc A.~Agrachev, Y.~Baryshnikov, and A.~Sarychev}, {\em Ensemble controllability by lie algebraic methods}, ESAIM: Control, Optimisation and Calculus of Variations, 22 (2016), pp.~921--938.

\bibitem{Agrachev2020Control}
{\sc A.~Agrachev and A.~Sarychev}, {\em Control in the spaces of ensembles of points}, SIAM Journal on Control and Optimization, 58 (2020), pp.~1579--1596.

\bibitem{Agrachev2022Control}
{\sc A.~Agrachev and A.~Sarychev}, {\em Control on the manifolds of mappings with a view to the deep learning}, Journal of Dynamical and Control Systems, 28 (2022), pp.~989--1008.

\bibitem{Agrachev2004Control}
{\sc A.~A. Agrachev and Y.~Sachkov}, {\em Control Theory from the Geometric Viewpoint}, vol.~2, Springer Science \& Business Media, 2004.

\bibitem{Ahlfors1979Complex}
{\sc L.~V. Ahlfors and L.~V. Ahlfors}, {\em Complex analysis}, vol.~3, McGraw-Hill New York, 1979.

\bibitem{Alvarez2024Interplay}
{\sc A.~Alvarez-Lopez, A.~H. Slimane, and E.~Zuazua}, {\em Interplay between depth and width for interpolation in neural odes}, Neural Networks, 180 (2024), p.~106640.

\bibitem{Benning2019Deep}
{\sc M.~Benning, E.~Celledoni, M.~J. Ehrhardt, B.~Owren, and C.-B. Sch{\"o}nlieb}, {\em Deep learning as optimal control problems: Models and numerical methods}, Journal of Computational Dynamics, 6 (2019), pp.~171--198.

\bibitem{Bensoussan2022Machine}
{\sc A.~Bensoussan, Y.~Li, D.~P.~C. Nguyen, M.-B. Tran, S.~C.~P. Yam, and X.~Zhou}, {\em Machine learning and control theory}, in Handbook of Numerical Analysis, vol.~23, Elsevier, 2022, pp.~531--558.

\bibitem{Bracewell1956Strip}
{\sc R.~N. Bracewell}, {\em Strip integration in radio astronomy}, Australian Journal of Physics, 9 (1956), pp.~198--217.

\bibitem{Brenier2003Approximation}
{\sc Y.~Brenier and W.~Gangbo}, {\em \${L}{\textasciicircum}p\$ {Approximation} of maps by diffeomorphisms}, Calculus of Variations and Partial Differential Equations, 16 (2003), pp.~147--164.

\bibitem{Brockett1972System}
{\sc R.~W. Brockett}, {\em System theory on group manifolds and coset spaces}, SIAM Journal on control, 10 (1972), pp.~265--284.

\bibitem{bruckner1997real}
{\sc A.~M. Bruckner, J.~B. Bruckner, and B.~S. Thomson}, {\em Real analysis}, Prentice-Hall, 1997.

\bibitem{Cai2023Achieve}
{\sc Y.~Cai}, {\em Achieve the {Minimum} {Width} of {Neural} {Networks} for {Universal} {Approximation}}, in The Eleventh International Conference on Learning Representations, 2023.

\bibitem{Cai2024Vocabulary}
{\sc Y.~Cai}, {\em Vocabulary for universal approximation: A linguistic perspective of mapping compositions}, in Forty-first International Conference on Machine Learning, 2024.

\bibitem{Chen2018Neural}
{\sc R.~T.~Q. Chen, Y.~Rubanova, J.~Bettencourt, and D.~Duvenaud}, {\em Neural ordinary differential equations}, in Proceedings of the 32nd {International} {Conference} on {Neural} {Information} {Processing} {Systems}, 2018, pp.~6572--6583.

\bibitem{Cheng2023Interpolation}
{\sc J.~Cheng, Q.~Li, T.~Lin, and Z.~Shen}, {\em Interpolation, approximation, and controllability of deep neural networks}, SIAM Journal on Control and Optimization, 63 (2025), pp.~625--649.

\bibitem{Cuchiero2020Deep}
{\sc C.~Cuchiero, M.~Larsson, and J.~Teichmann}, {\em Deep {Neural} {Networks}, {Generic} {Universal} {Interpolation}, and {Controlled} {ODEs}}, SIAM Journal on Mathematics of Data Science, 2 (2020), pp.~901--919.

\bibitem{Cybenkot1989Approximation}
{\sc G.~Cybenkot}, {\em Approximation by superpositions of a sigmoidal function}, Mathematics of Control, Signals and Systems, 2 (1989), pp.~303--314.

\bibitem{Duan2025Minimal}
{\sc Y.~Duan and Y.~Cai}, {\em A minimal control family of dynamical systems for universal approximation}, IEEE Transactions on Automatic Control,  (2025), pp.~1--12.

\bibitem{Duan2022Vanilla}
{\sc Y.~Duan, L.~Li, G.~Ji, and Y.~Cai}, {\em Vanilla feedforward neural networks as a discretization of dynamic systems}, Journal of Scientific Computing, 101 (2024), p.~82.

\bibitem{Weinan2017A}
{\sc W.~E}, {\em A proposal on machine learning via dynamical systems}, Communications in Mathematics and Statistics, 5 (2017), pp.~1--11.

\bibitem{Gray2004Steiner}
{\sc A.~Gray}, {\em Steiner's Formula}, Birkh{\"a}user Basel, Basel, 2004, pp.~209--229.

\bibitem{haber2017stable}
{\sc E.~Haber and L.~Ruthotto}, {\em Stable architectures for deep neural networks}, Inverse problems, 34 (2017), p.~014004.

\bibitem{He2016Deep}
{\sc K.~He, X.~Zhang, S.~Ren, and J.~Sun}, {\em Deep residual learning for image recognition}, in Proceedings of the IEEE conference on computer vision and pattern recognition, 2016, pp.~770--778.

\bibitem{He2016Identity}
{\sc K.~He, X.~Zhang, S.~Ren, and J.~Sun}, {\em Identity mappings in deep residual networks}, in European conference on computer vision, 2016, pp.~630--645.

\bibitem{Hornik1989Multilayer}
{\sc K.~Hornik, M.~Stinchcombe, and H.~White}, {\em Multilayer feedforward networks are universal approximators}, Neural networks, 2 (1989), pp.~359--366.

\bibitem{Jurdjevic1997Geometric}
{\sc V.~Jurdjevic}, {\em Geometric control theory}, Cambridge university press, 1997.

\bibitem{Kang2022Feedforward}
{\sc W.~Kang and Q.~Gong}, {\em Feedforward neural networks and compositional functions with applications to dynamical systems}, SIAM Journal on Control and Optimization, 60 (2022), pp.~786--813.

\bibitem{Leshno1993Multilayer}
{\sc M.~Leshno, V.~Y. Lin, A.~Pinkus, and S.~Schocken}, {\em Multilayer feedforward networks with a nonpolynomial activation function can approximate any function}, Neural Networks, 6 (1993), pp.~861--867.

\bibitem{Li2018optimal}
{\sc Q.~Li and S.~Hao}, {\em An optimal control approach to deep learning and applications to discrete-weight neural networks}, in International Conference on Machine Learning, PMLR, 2018, pp.~2985--2994.

\bibitem{Li2022Deep}
{\sc Q.~Li, T.~Lin, and Z.~Shen}, {\em Deep learning via dynamical systems: {An} approximation perspective}, Journal of the European Mathematical Society, 25 (2022), pp.~1671--1709.

\bibitem{Liu2019Deep}
{\sc G.-H. Liu and E.~A. Theodorou}, {\em Deep learning theory review: An optimal control and dynamical systems perspective}, arXiv preprint arXiv:1908.1092,  (2019).

\bibitem{Maas2013Rectifier}
{\sc A.~L. Maas, A.~Y. Hannun, and A.~Y. Ng}, {\em Rectifier nonlinearities improve neural network acoustic models}, in International Conference on Machine Learning, 2013.

\bibitem{Magnus2019Matrix}
{\sc J.~R. Magnus and H.~Neudecker}, {\em Matrix differential calculus with applications in statistics and econometrics}, John Wiley \& Sons, 2007.

\bibitem{McLachlan2002Splitting}
{\sc R.~I. McLachlan and G.~R.~W. Quispel}, {\em Splitting methods}, Acta Numerica, 11 (2002), pp.~341--434.

\bibitem{Montgomery2002tour}
{\sc R.~Montgomery}, {\em A tour of subriemannian geometries, their geodesics and applications}, no.~91, American Mathematical Soc., 2002.

\bibitem{Muger2019Notes}
{\sc M.~Müger}, {\em Notes on the theorem of {Baker-Campbell-Hausdorff-Dynkin}}, Radboud University: Nijmegen, The Netherlands, 35 (2019).

\bibitem{pinkus1999approximation}
{\sc A.~Pinkus}, {\em Approximation theory of the mlp model in neural networks}, Acta numerica, 8 (1999), pp.~143--195.

\bibitem{Ruiz2022Interpolation}
{\sc D.~Ruiz-Balet, E.~Affili, and E.~Zuazua}, {\em Interpolation and approximation via momentum resnets and neural odes}, Systems \& Control Letters, 162 (2022), p.~105182.

\bibitem{Ruiz-Balet2023Neural}
{\sc D.~Ruiz-Balet and E.~Zuazua}, {\em Neural {ODE} control for classification, approximation and transport}, SIAM Review, 65 (2023), pp.~735--773.

\bibitem{Stein2011Fourier}
{\sc E.~M. Stein and R.~Shakarchi}, {\em Fourier analysis: an introduction}, vol.~1, Princeton University Press, 2011.

\bibitem{Steiner1840parallele}
{\sc J.~Steiner}, {\em \"uber parallele fl\"achen}, Monatsberichte der K\"oniglich Preu{\ss}ischen Akademie der Wissenschaften zu Berlin,  (1840), pp.~114--118.
\newblock Reprinted in Gesammelte Werke, Vol.~2 (Reimer, Berlin, 1882), pp.~171--176.

\bibitem{stone1948generalized}
{\sc M.~H. Stone}, {\em The generalized weierstrass approximation theorem}, Mathematics Magazine, 21 (1948), pp.~237--254.

\bibitem{Sussmann1972Controllability}
{\sc H.~J. Sussmann and V.~Jurdjevic}, {\em Controllability of nonlinear systems.}, Journal of Differential Equations, 12 (1972).

\bibitem{Tabuada2022Universal}
{\sc P.~Tabuada and B.~Gharesifard}, {\em Universal approximation power of deep residual neural networks through the lens of control}, IEEE Transactions on Automatic Control,  (2022), pp.~1--14.

\bibitem{Yoshida1990Construction}
{\sc H.~Yoshida}, {\em Construction of higher order symplectic integrators}, Physics Letters A, 150 (1990), pp.~262--268.

\end{thebibliography}

\end{document}